\documentclass[11pt]{rmmcart}
\usepackage[margin=2.5cm]{geometry}
\usepackage{amsmath, amsthm, amssymb}
\usepackage{epsfig}
\usepackage{rawfonts}
\usepackage{enumerate}
\usepackage{graphics}
\usepackage{verbatim}
\usepackage{xspace}
     \usepackage{graphicx}

\usepackage{pgf,tikz,pgfplots}
\usepackage{mathrsfs}
\usetikzlibrary{arrows}
\usepackage{amsmath}
\usepackage{amsfonts}
\usepackage{amssymb}
\usepackage{amsthm}
\usepackage{graphicx}
\usepackage{booktabs}
\usepackage{caption}
\usepackage{listings}
\usepackage{setspace}
\usepackage[mathscr]{eucal}
\usepackage{pgfplots}
\usepackage{hyperref}
\usepackage{wrapfig}
\usepackage{floatflt,epsfig}
\usepackage{ dsfont }
\usepackage{amscd}
\usepackage{tikz-cd}
\usepackage{fancyhdr}
\usepackage[all]{xy}
\usepackage{latexsym}
\usepackage{amscd}
\usepackage{pifont}
\usepackage{eufrak}
\usepackage{subfig}

\newcommand{\rank}{\mathop{\rm rank}\nolimits}

\newcommand{\numberset}{\mathbb}
\newcommand{\N}{\numberset{N}}
\newcommand{\Z}{\numberset{Z}}

\newcommand{\R}{\numberset{R}}

\def\NN{{\mathbb N}}
\def\ZZ{{\mathbb Z}}

\newcommand{\cC}{\mathcal{C}}

\newcommand{\cP}{\mathcal{P}}
\newcommand{\cH}{\mathcal{H}}

\newcommand{\cB}{\mathcal{B}}

\newcommand{\cF}{\mathcal{F}}
\newcommand{\cW}{\mathcal{W}}
\newcommand{\cR}{\mathcal{R}}
\newcommand{\cS}{\mathcal{S}}

\theoremstyle{plain}

\theoremstyle{theorem}

\newtheorem{defn}{Definition}[section]
\newtheorem{prop}[defn]{Proposition}
\newtheorem{thm}[defn]{Theorem}
\newtheorem{lemma}[defn]{Lemma}

\newtheorem{coro}[defn]{Corollary}

\newtheorem{rmk}[defn]{Remark}

\theoremstyle{remark}

\begin{document}
		
		\title[PRIMALITY OF CLOSED PATH POLYOMINOES]{PRIMALITY OF CLOSED PATH POLYOMINOES}
		\author{CARMELO CISTO}
		\address{Universit\'{a} di Messina, Dipartimento di Scienze Matematiche e Informatiche, Scienze Fisiche e Scienze della Terra\\
			Viale Ferdinando Stagno D'Alcontres 31\\
			98166 Messina, Italy}
		\email{carmelo.cisto@unime.it}

		\author{FRANCESCO NAVARRA}
		\address{Universit\'{a} di Messina, Dipartimento di Scienze Matematiche e Informatiche, Scienze Fisiche e Scienze della Terra\\
			Viale Ferdinando Stagno D'Alcontres 31\\
			98166 Messina, Italy}
		\email{francesco.navarra@unime.it}

		

		\begin{abstract}
	 In this paper we introduce a new class of polyominoes, called \emph{closed paths}, and we study the primality of their associated ideal. Inspired by an existing conjecture that characterizes the primality of a polyomino ideal by nonexistence of zig-zag walks, we classify all closed paths which do not contain zig-zag walks, and we give opportune toric representations of the associated ideals. To support the conjecture we prove that having no zig-zag walks is a necessary and sufficient condition for the primality of the associated ideal of a closed path. Finally, we present some classes of prime polyominoes viewed as generalizations of closed paths.\\
		
		\end{abstract}

		\maketitle
		
	\section{Introduction}
	
	Polyominoes are plane figures, made up of squares of the same size joined edge by edge. They appeared for the first time in recreational mathematics and combinatorics about sixty years ago and they are studied especially in tiling problems of the plane. In 2012 Ayesha A. Qureshi connected polyominoes to Commutative Algebra, assigning to every polyomino $\cP$ the ideal of inner 2-minors, called a \textit{polyomino ideal}, denoted by $I_{\cP}$(\cite{Qureshi}). In literature there are many studies about the ideals of $t$-minors of an $m \times n$ matrix of indeterminates, for any integer 
	$1\leq t\leq \min\{m,n\}$, and the ideals
	generated by all $t$-minors of a one or two sided ladder (\cite{conca1}, \cite{conca2}, \cite{conca3}).
	The ideals of adjacent 2-minors are discussed in several papers (\cite{adiajent1},\cite{adiajent3},\cite{adjent 2}) as well as the ideals generated by an arbitrary set
	of $2$-minors in a $2\times n$ matrix (\cite{2.n}).
	The class of polyomino ideals generalizes the class of the ideals generated by 2-minors of $m\times n$ matrices. \\
	Our goal is to investigate the main algebraic properties of $I_{\cP}$, depending on the shape of $\cP$; in particular it is interesting to study its primality. We know that polyomino ideals attached to simple polyominoes are prime ideals. Roughly speaking, a simple polyomino is a polyomino without holes. The primality of simple polyominoes has been proved in \cite{Simple equivalent balanced}, showing that simple and balanced polyominoes are equivalent and using the fact that a polyomino ideal associated to a balanced one is prime (see \cite{def balanced}). Independently of this, in \cite{Simple are prime} it has been shown that polyomino ideals associated to simple polyominoes are prime, by identifying their quotient ring with the toric ring of a weakly chordal graph. The study is applied to multiply connected polyominoes, that are polyominoes with one or more holes. In \cite{Not simple with localization} and \cite{Shikama}, the authors discuss a family of prime polyominoes, obtained by removing a convex polyomino from a rectangle in $\NN^2$. It is not an easy task to give a complete classification of all polyominoes whose ideal is prime. In \cite{Trento} the authors give an interesting tool, which seems to be very useful to characterize the primality of multiply connected polyominoes. They define a particular sequence of inner intervals of $\cP$, called a \textit{zig-zag walk}, and they prove that its nonexistence in $\cP$ is a necessary condition to the primality of $I_{\cP}$. Moreover they show that it is a sufficient condition for polyominoes made up of at most fourteen cells, by a computational method. It seems that nonexistence of zig-zag walks in a polyomino should characterize its primality. \\
	In this paper we study the primality of other classes of polyomino ideals; in particular we introduce a class of polyominoes, namely \emph{closed paths}, in which having no zig-zag walks is a sufficient condition for the associated ideal to be prime. We mention that such polyominoes were introduced independently also in [13], where the authors call them \textit{thin cycles}. In particular in [13, Corollary 3.6] the authors give another interesting sufficient condition for the primality of such polyominoes, using Gr\"obner basis theory. 
	In Section \ref{Section Introductioon} we introduce the preliminary notions and some useful tools. In Section \ref{Section Closed path} we define the particular class of multiply connected polyominoes which we call \textit{closed paths}, and we give two sufficient geometric conditions in order that they do not contain zig-zag walks. In fact we introduce the L-configuration, that consists of a path of five cells $A_1, \dots, A_5$ such that the two blocks $A_1, A_2, A_3$ and $A_3, A_4, A_5$ go in  orthogonal directions, and we prove that if a closed path has an L-configuration, then it does not contain zig-zag walks. Moreover we define a ladder with at least three steps in a closed path and we show that having this structure is a sufficient condition to have no zig-zag walks. In Sections \ref{Section L-conf toric} and \ref{Section toric ladder} we give a toric representation of a closed path with an L-configuration or a ladder respectively, using a method similar to Shikama's one in \cite{Shikama} and attaching the hole variable only to a particular set of vertices of the L-configuration or ladder. In Section \ref*{Section main result} we present the main result of the paper. At first we show that having an L-configuration or a ladder with at least three steps is a necessary and sufficient condition in order to have no zig-zag walks for a closed path. This result characterizes the structure of closed paths which contain no zig-zag walks, and  makes possible to prove that the conjecture in \cite{Trento} is true for such a class.  
	At the end of this work we study some particular classes of prime polyominoes, that we can build using paths and simple polyominoes. They are a weak generalization of closed paths. We give some sufficient conditions for the primality of the related ideals. Necessary conditions for primality, and so proving the conjecture of \cite{Trento} for such classes, are difficult to find out and are left here as open questions.
		
	\section{Basics on polyominoes and polyomino ideals}\label{Section Introductioon}
We consider the natural partial order on $\R^2$: given $(i,j),(k,l)\in \R^2$, we say $(i,j)\leq(k,l)$ if $i\leq k$ and $j\leq l$. Let $a=(i,j),b=(k,l)\in\Z^2$. The set $[a,b]=\{(r,s)\in \Z^2: i\leq r\leq k,\ j\leq s\leq l \}$ is called an \textit{interval} of $\Z^2$. We define the \textit{closure} of $[a,b]$ the set $\overline{[a,b]}=\{x\in \R^2:a\leq x\leq b \}$. If $i< k$ and $j< l$, we say that $[a,b]$ is a \textit{proper} interval. The elements $a, b$ are called the \textit{diagonal corners} and $c=(i,l)$, $d=(k,j)$ the \textit{anti-diagonal corners} of $[a,b]$. If $j=l$ (or $i=k$) we say that $a$ and $b$ are in \textit{horizontal} (or \textit{vertical}) \textit{position}. An elementary interval of the form $C=[a,a+(1,1)]$ is a \textit{cell} with \textit{lower left corner} $a$. The elements $a$, $a+(0,1)$, $a+(1,0)$ and $a+(1,1)$ are called the \textit{vertices} or \textit{corners} of $C$ and the sets $\{a,a+(1,0)\}$, $\{a+(1,0),a+(1,1)\}$, $\{a+(0,1),a+(1,1)\}$ and $\{a,a+(0,1)\}$ are called the \textit{edges} of $C$. We denote the set of the vertices and the edges of $C$ respectively by $V(C)$ and $E(C)$. \\
Let $C$ and $D$ be two distinct cells of $\ZZ^2$. A \textit{walk} from $C$ to $D$ is a sequence $\cC:C=C_1,\dots,C_m=D$ of cells of $\ZZ^2$ such that $C_i \cap C_{i+1}$ is an edge of $C_i$ and $C_{i+1}$ for $i=1,\dots,m-1$. If in addition $C_i \neq C_j$ for all $i\neq j$, then $\cC$ is called a \textit{path} from $C$ to $D$. If $\cC_1:A_1,\dots,A_m$ and $\cC_2:B_1,\dots,B_n$ are two walks such that $A_m=B_1$, then the union of $\cC_1$ and $\cC_2$ is defined by the walk $A_1,\dots,A_{m-1},A_m,B_2,\dots,B_n$ and it is denoted by $\cC_1\cup \cC_2$.

	\begin{rmk}\rm \label{Remark: walk diventa path}
		
		In general, a walk $\cC:C=C_1,\dots,C_m=D$ from $C$ to $D$ contains a path between $C$ and $D$, that is there exists a path $\cF$ from $C$ to $D$ such that every cell of $\cF$ is a cell of $\cC$.  It can be proved by induction on the number $m$ of cells of $\cC$. If $m=2$, then $\cC:C=C_1,C_2=D$ is obviously a path. Let $m>2$ and suppose that any walk from $C$ to $D$ consisting of $k$ cells, with $k<m$, contains a path between these cells. Suppose that not all the cells of $\cC$ are distinct, so there exist $i,j\in\{1,\dots,m\}$ such that $C_i=C_j$, with $j>i$. Consider the sequence $\cC':C=C_1,\dots,C_{i-1},C_j,\dots,C_m=D$, consisting of all the cells in $\cC$ except $C_{i},C_{i+1},\dots,C_{j-1}$.
		$\cC'$ is a walk from $C$ to $D$ having less than $m$ cells, so applying the inductive hypothesis on $\cC'$ we have a path from $C$ to $D$ contained in $\cC'$, that is contained also in $\cC$. 
	\end{rmk} 
\noindent	Let $\cP$ be a non-empty collection of
cells in $\Z^2$. We denote the set of the vertices of $\cP$ by $V(\cP)=\bigcup_{C\in \cP}V(C)$ and the set of the edges of $\cP$ by $E(\cP)=\bigcup_{C\in \cP}E(C)$. Let $C$ and $D$ be two cells of $\cP$. We say that $C$ and $D$ are \textit{connected} if there exists a walk $\cC:C=C_1,\dots,C_m=D$ such that $C_i \in \cP$ for all $i=1,\dots,m$. We denote by $(a_i,b_i)$ the lower left corner of $C_i$ for all $i=1,\dots,m$ and we observe that a walk can change direction in one of the following ways:
	\begin{enumerate}
		\item North, if $(a_{i+1}-a_i,b_{i+1}-b_i)=(0,1)$ for some $i=1,\dots,m-1$;
		\item South, if $(a_{i+1}-a_i,b_{i+1}-b_i)=(0,-1)$ for some $i=1,\dots,m-1$;
		\item East, if $(a_{i+1}-a_i,b_{i+1}-b_i)=(1,0)$ for some $i=1,\dots,m-1$;
		\item West, if $(a_{i+1}-a_i,b_{i+1}-b_i)=(-1,0)$ for some $i=1,\dots,m-1$.
	\end{enumerate}
A non-empty, finite collection $\cP$ of cells in $\Z^2$  is called \textit{polyomino} if any two cells of $\cP$ are connected.
	\begin{figure}[h!]
		\centering
\includegraphics[scale=0.45]{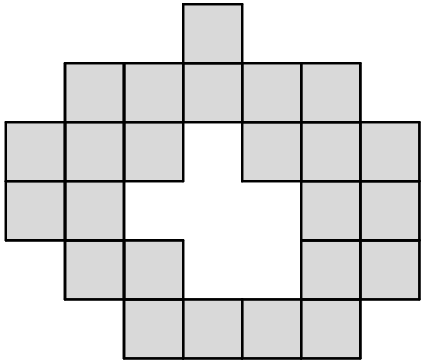}
		\caption{A polyomino.}
		\label{A polyomino}
	\end{figure}

\noindent Let $\cP$ be a polyomino. We say that $\cP$ is \textit{simple} if for any two cells $C$ and $D$ of $\Z^2$, not in $\cP$, there exists a path $\cC: C=C_1,\dots,C_m=D$ such that $C_i\notin \cP$  for all $i=1,\dots,m$.  For example, the polyomino in Figure \ref{A polyomino} is not simple. A finite collection of cells $\cH$ not in $\cP$ is a \textit{hole} of $\cP$ if any two cells $F$ and $G$ of $\cH$ are connected by a path $\cF:F=F_1,\dots,F_t=G$ such that $F_j\in \cH$ for all $j=1,\dots,t$ and $\cH$ is maximal with respect to set inclusion. We observe that each hole of a polyomino $\cP$ is a simple polyomino and $\cP$ is simple if and only if it does not have any hole. Moreover, it is easy to see that a non-simple polyomino has a finite number of holes.  We say that a cell $E$ of $\Z^2$ is \textit{external} to $\cP$ if it satisfies one of the two following conditions: $E \notin \cP \cup \mathcal{H}_1\cup \cdots \cup \mathcal{H}_n$ if $\cP$ is a non-simple polyomino and $\mathcal{H}_1,\ldots,\mathcal{H}_n$ are the holes of $\cP$, or $E\notin \cP$ if $\cP$ is a simple polyomino. The set of the cells of $\Z^2$ external to $\cP$ is called the \textit{exterior} of $\cP$. If $\mathcal{U}$ is the exterior of $\cP$, then we observe that any two cells of $\mathcal{U}$ are connected in $\mathcal{U}$. A proper interval $[a,b]$ is called an \textit{inner interval} of $\cP$ if all cells of $[a,b]$ belong to $\cP$. An interval $[a,b]$ with $a=(i,j)$, $b=(k,j)$ and $i<k$ is called a \textit{horizontal edge interval} of $\cP$ if the sets $\{(\ell,j),(\ell+1,j)\}$ are edges of cells of $\cP$ for all $\ell=i,\dots,k-1$. In addition, if $\{(i-1,j),(i,j)\}$ and $\{(k,j),(k+1,j)\}$ do not belong to $E(\cP)$, then $[a,b]$ is called a \textit{maximal} horizontal edge interval of $\cP$. Similarly, we define \textit{vertical edge intervals} and \textit{maximal} vertical edge intervals. We say that an edge of a cell of $\cP$ is a \textit{border edge} if it is not an edge of any other cell of $\cP$. A \textit{horizontal border edge of $\cP$} is defined to be an horizontal edge interval of $\cP$ consisting of border edges of cells of $\cP$. Similarly we define the \textit{vertical border edge of $\cP$}. The union of the closures of the border edges of $\cP$ is called \textit{perimeter} of $\cP$. \\
Let $A$ and $B$ be two cells of $\Z^2$ with lower left corners $(i,j)$ and $(k,l)$, respectively. The \textit{cell interval}, denoted by $[A,B]$, is the set of the cells of $\Z^2$ with lower left corner $(r,s)$ for $i\leqslant r\leqslant k$ and $j\leqslant s\leqslant l$. If $(i,j)$ and $(k,l)$ are in horizontal position, we say that the cells $A$ and $B$ are in horizontal position. Similarly, we define two cells in vertical position. Let $A$ and $B$ be two cells of $\cP$ in vertical or horizontal position. The cell interval $[A,B]$ is called a
 \textit{block of $\cP$ of length n} if any cell $C$ of $[A,B]$ belongs to $\cP$ and $n$ is the number of cells in $[A,B]$. The cells $A,B$ are called the \emph{extremal cells} of the block $[A,B]$. The block $[A,B]$ is defined  to be \textit{maximal} if there does not exist any block $[A',B']$ of $\cP$ such that $[A,B]\subset [A',B']$. Moreover if $A$ and $B$ are in vertical (resp. horizontal) position, then $[A,B]$ is also called a \textit{maximal vertical (resp. horizontal) block} of $\cP$. Observe that to each interval $I$ of $\Z^2$ we can attach obviously a cell interval of $\Z^2$, which we indicate by $\cP(I)$.\\ 
\noindent	We follow \cite{Trento} and we call a \textit{zig-zag walk} of $\cP$ a sequence $\cW:I_1,\dots,I_\ell$ of distinct inner intervals of $\cP$ where, for all $i=1,\dots,\ell$, the interval $I_i$ has either diagonal corners $v_i$, $z_i$ and anti-diagonal corners $u_i$, $v_{i+1}$ or anti-diagonal corners $v_i$, $z_i$ and diagonal corners $u_i$, $v_{i+1}$, such that:
	\begin{enumerate}
		\item $I_1\cap I_\ell=\{v_1=v_{\ell+1}\}$ and $I_i\cap I_{i+1}=\{v_{i+1}\}$, for all $i=1,\dots,\ell-1$;
		\item $v_i$ and $v_{i+1}$ are on the same edge interval of $\cP$, for all $i=1,\dots,\ell$;
		\item for all $i,j\in \{1,\dots,\ell\}$ with $i\neq j$, there exists no inner interval $J$ of $\cP$ such that $z_i$, $z_j$ belong to $J$.
	\end{enumerate}
	\begin{figure}[h]
	\centering
	\includegraphics[scale=0.55]{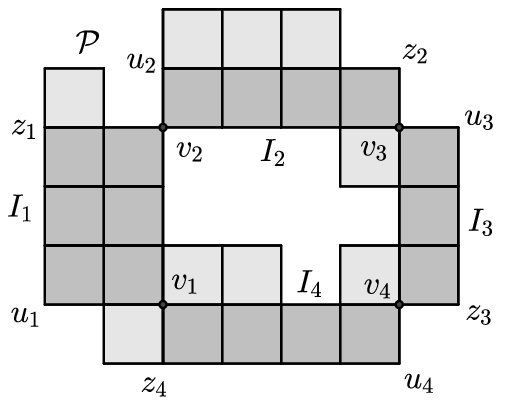}
	\caption{An example of a zig-zag walk of $\cP$.}
	\end{figure}

\noindent	Let $\cP$ be a polyomino. Let $K$ be a field and $S=K[x_v\mid v\in V(\cP)]$. To each proper interval $[a,b]$, where $a$ and $b$ are the diagonal corners, $c$ and $d$ the anti-diagonal ones, we associate the binomial $x_ax_b-x_cx_d$. If $[a,b]$ is an inner interval, the binomial $x_ax_b-x_cx_d$ is called an \textit{inner 2-minor} of $\cP$. The ideal $I_{\cP}\subset S$ generated by all the inner 2-minors of $\cP$ is called the \textit{polyomino ideal of $\cP$} and $K[\cP] = S/I_{\cP}$ the coordinate ring of $\cP$.\\
We conclude this section recalling some notations and definitions contained in \cite{Shikama}. Moreover, we provide a more general version of \cite[Lemma 2.2]{Shikama}, which is very useful for this work. A binomial $f=f^+-f^-$ in a binomial ideal $J\subset S$ is called \textit{redundant} if it can be expressed as a linear combination of binomials in $J$ of lower degree. A binomial is called \textit{irredundant} if it is not redundant. Moreover, we denote by $V^{+}_f$ the set of the vertices $v$, such that $x_v$ appears in $f^+$, and by $V^{-}_f$ the set of the vertices $v$, such that $x_v$ appears in $f^-$.
	
	\begin{lemma}\label{Lemma shikama1}
	Let $\cP$ be a polyomino and $\phi: S\rightarrow T$ a ring homomorphism with $T$ an integral domain. Let $J=\ker \phi$ and $f=f^+-f^-$ be a binomial in $J$ with $\deg f\geq3$. Suppose that: \begin{itemize}
	    \item $I_\cP \subseteq J$;
	    \item $\phi(x_r)\neq 0$ for all $r\in V(\cP)$. 
	    \end{itemize}
	    If there exist three vertices $p,q\in V^+_{f}$ and $r\in V^-_{f}$ such that $p,q$ are diagonal (respectively  anti-diagonal) corners of an inner interval and $r$ is one of the anti-diagonal (respectively diagonal) corners of the inner interval, then $f$ is redundant in $J$.
    \end{lemma}
\begin{proof}
Let $I$ be the inner interval of $\cP$, such that $p$,$q$ are the diagonal vertices and $r$ is an anti-diagonal one. We denote by $s$ the other corner of $I$. We set $f_I=x_px_q-x_rx_s$ and $f_J=x_s\frac{f^+}{x_px_q}-\frac{f^-}{x_r}$. We have:
$$
    f=f^+-f^-=\Big(x_px_q-x_rx_s\Big)\frac{f^+}{x_px_q}+x_r\Bigg(x_s\frac{f^+}{x_px_q}-\frac{f^-}{x_r}\Bigg)=f_I\frac{f^+}{x_px_q}+x_rf_J.
$$
 Since $I_\cP \subseteq J$, it follows that $f_I\in J$. Since $f, f_I\in J$, we have $x_r f_{J}\in J$. Moreover $\phi(x_r)\neq 0$ and $T$ is a domain, so $f_{J}\in J$.  We observe that $\deg f_I$ and $\deg f_J$ are strictly less than $\deg f$, so we have the desired conclusion.
\end{proof}

\noindent Observe that the same claim of the previous result holds also if $p,q\in V^-_f$ and $r\in V^+_f$, by the same argument.
	
	\section{Closed paths} \label{Section Closed path}
\begin{defn}\label{def closed path}\rm 
		A  polyomino $\cP$ is called a \textit{closed path} if it is a sequence of cells $A_1,\dots,A_n, A_{n+1}$, $n>5$, such that:
	\begin{enumerate}
		\item $A_1=A_{n+1}$;
		\item $A_i\cap A_{i+1}$ is a common edge, for all $i=1,\dots,n$;
		\item $A_i\neq A_j$, for all $i\neq j$ and $i,j\in \{1,\dots,n\}$;
		\item For all $i\in\{1,\dots,n\}$ and for all $j\notin\{i-2,i-1,i,i+1,i+2\}$ then $A_i\cap A_j=\emptyset$, where $A_{-1}=A_{n-1}$, $A_0=A_n$, $A_{n+1}=A_1$ and $A_{n+2}=A_2$. 
	\end{enumerate}
	
	\begin{figure}[h]
	\centering
	\includegraphics[scale=0.75]{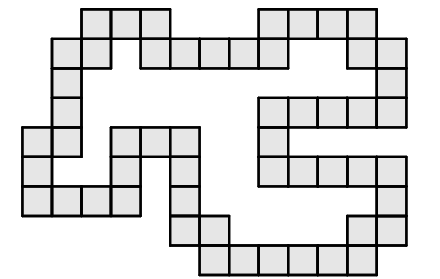}
	\caption{A closed path.}
	\end{figure}
\end{defn}

\noindent Intuitively, a closed path is a path in which the two ends meet and the cells have a common edge only with the previous and next ones. Roughly speaking, it is similar to a pearl necklace on a table. The assumption $n>5$ is 
not restrictive, in fact it is known that all polyominoes with less than 6 cells are  simple polyominoes (see for instance \cite{golomb}), so they are well known for what concerns the primality of $I_\cP$ and other properties of such an ideal.

\begin{rmk}\label{simple-remark}\rm
Let $\cP$ be a closed path and $A_1,A_2,\dots,A_n,A_{n+1}=A_1$ the sequence of cells of $\cP$ having the properties in Definition \ref{def closed path}. Let $i\in \{1,\ldots,n\}$ and consider the cells $A_{i-1},A_i,A_{i+1}$. Up to reflections or rotations we have only one of the two arrangements described in Figure~\ref{fig:simple} (A) and (B). 
\begin{figure}[h]
\subfloat[][]{\includegraphics[scale=0.8]{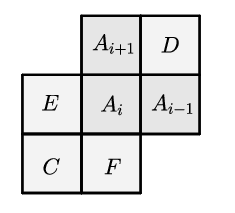}}\qquad\qquad  \qquad
\subfloat[][]{\includegraphics[scale=0.8]{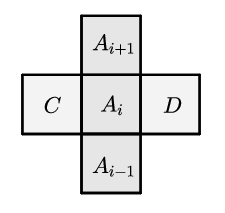}}
\caption{}
\label{fig:simple}
\end{figure}
\begin{enumerate}
\item Referring to Figure~\ref{fig:simple}(A), without loss of generality, we can suppose that $i=1$, so $i+1=2$ and $i-1=n$, otherwise it suffices to rename the indices. 
We prove that $C,D,E,F$ do not belong to $\cP$. Suppose that $E$ belongs to $\cP$. 
Since $E\cap A_1\neq \emptyset$, from condition (4) of Definition~\ref{def closed path} we have that  
$E=A_3$ or $E=A_{n-1}$, which contradicts (2) of Definition~\ref{def closed path}, because $E\cap A_{2}$ and $E\cap A_{n}$ are not edges. So $E$ does not belong to $\cP$. The same holds for the cells $C$ and $F$ by similar arguments. Moreover, from condition (4) of Definition~\ref{def closed path} it follows that $D$ is not in $\cP$.  
By similar arguments it is possible to show that the cells $C$ and $D$ as in Figure~\ref{fig:simple}(B) do not belong to $\cP$. \\


\item Observe that for every cell $H$ not belonging to $\cP$ and for every cell $A$ belonging to $\cP$ 
 and not placed as the cell $A_i$ in Figure~\ref{fig:simple} (A) there exists a path of cells $H=H_1,\dots,H_m$ not belonging to $\cP$ such that $H_m\cap A$ is an edge of $H_m$ and $A$. In fact it is possible to consider a walk $\cC_1:H=F_1,\dots,F_r=G$ of cells not in $\cP$ linking $H$ to a cell $G$ not in $\cP$ and having an edge in common with a cell $A_k$ of $\cP$. We may assume that $k=1$, otherwise it suffices to rename the indices. If $A=A_1$, then $\cC_1$ is a walk of cells not in $\cP$ such that $F_r\cap A$ is an edge of $F_r$ and $A$, and by Remark \ref*{Remark: walk diventa path} we obtain a desired path. If $A\neq A_1$, then we can consider another walk $\cC_2:G=G_1,\dots,G_t$ such that $G_j \notin \cP$ for all $j=1,\dots,t$ and $G_t\cap A$ is an edge of $G_t$ and $A$, obtained travelling along the perimeter of $\cP$ with the condition (2) of the Definition~\ref{def closed path} and using the point (1) of this Remark. Considering the walk $\cC_1\cup\cC_2$, we have a desired path by Remark \ref*{Remark: walk diventa path}.
\end{enumerate}
\end{rmk}

 	\begin{lemma}\label{A closed path contains a block}
		Let $\cP$ be a closed path. Then $\cP$ contains a block of length at least 3.
	\end{lemma}
	\begin{proof}
		We suppose that $\cP$ does not contain any block of length $n\geqslant 3$. We fix a cell $A$ of $\cP$ with lower left corner $a$. After a shift of coordinates, we may assume that $a=(1,1)$. Since $\cP$ is a closed path, there exists a cell $A_2$, which has an edge in common with $A$. We may assume that the lower left corner of $A_2$ is $a_2=(2,1)$.
		$\cP$ is a closed path, then there exists a cell $B_2$, different from $A$, such that $A_2\cap B_2$ is an edge of $A_2$ and $B_2$. If the lower left corner of $B_2$ is $(3,1)$, then $\{A, A_2, B_2\}$ is a block of length three, a contradiction. We may assume that the lower left corner of $B_2$ is $b_2=(2,2)$. Continuing these arguments, we find a sequence of cells of $\cP$, namely $A, A_2, B_2, \dots, A_m, B_m, \dots$, where the lower left corners of $A_m$ and $B_m$ are respectively $a_m=(m,m-1)$ and $b_m=(m,m)$, for all $m\geqslant 2$. Since $\cP$ is a closed path, there exists  $\overline{m}\in \N\setminus \{0,1\}$ such that $A_{\overline{m}}=A$ or $B_{\overline{m}}=A$, that is $a_{\overline{m}}=(1,1)$ or $b_{\overline{m}}=(1,1)$. It is a contradiction because $a_{m}>(1,1)$ and $b_{m} >(1,1)$, for all $m\geqslant 2$.
	\end{proof}
 \noindent According to \cite{Simple equivalent balanced}, we recall that a \textit{rectilinear polygon} is a polygon whose edges meet orthogonally and it is called \textit{simple} if there does not exist any self-intersection. In particular if $\mathfrak{C}$ is a rectilinear polygon, then the area bounded by $\mathfrak{C}$ is called the \textit{interior} of $\mathfrak{C}$.
\begin{prop}\label{P is a not closed path}
	Let $\cP:A_1,\dots,A_n, A_{n+1}$ be a closed path. Then the following hold:
	\begin{enumerate}
	\item $\cP$ is a non-simple polyomino.
	\item $\cP$ has a unique hole.
	\item Let $\cP'$ be the polyomino consisting of all the cells of $\cP$ except $A_i,A_{i+1},\ldots,A_{i+r}$ for some $i\in \{1,\ldots,n\}$ and $1\leq r<n-1$, where all indices are reduced modulo $n$. Then $\cP'$ is a simple polyomino. 
\end{enumerate}
\end{prop}
\begin{proof}
1) 
 Firstly we show that there exist two cells not belonging to $\cP$ and a simple rectilinear polygon $\mathfrak{C}$, consisting of the union of the closures of certain border edges of $\cP$, such that the two cells are both neither in the interior of $\mathfrak{C}$ nor in the exterior of $\mathfrak{C}$. 
Consider a change of direction of $\cP$ consisting of the cells $R$, $S$ and $T$ and we do opportune rotations of $\cP$ in order to have $\{R,S,T\}$ as in Figure \ref{dim05} (A). We set $S=A_1$ and, walking clockwise along the path, we label the cells of $\cP$ increasingly from $A_1$ to $A_n$. It is not restrictive to assume $R=A_n$ and $T=A_2$. Observe that a such labelling induces a natural orientation along the perimeter of $\cP$. Let $\mathfrak{C}$ be the union of the closures of the border edges of $\cP$ having the following property: if $r$ is a border edge of a cell $A_i$ then $\overline{r}\in \mathfrak{C}$ if it has the cell $A_i$ on its right with respect to the fixed orientation on the perimeter of $\cP$.
 	\begin{figure}[h]
 	\centering
 	\subfloat[][]{\includegraphics[scale=0.8]{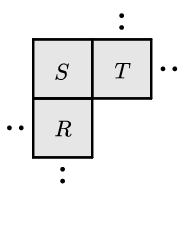}}	\qquad\qquad
 	\subfloat[][The arrows indicate the clockwise orientation of $\mathfrak{C}$.]{\includegraphics[scale=1.2]{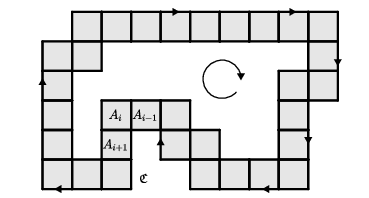}}
   \caption{}
 		\label{dim05}
 	\end{figure}
 	We prove that $\mathfrak{C}$ is a simple rectilinear polygon. Observe that $\mathfrak{C}$ is the union of orthogonal line segments by construction, so if $\mathfrak{C}$ is a polygon then it is also rectilinear. We show firstly that $\mathfrak{C}$ is a polygon. We denote by $r_1$ the border edge of $A_1$ having a vertex in common with $A_2$. Let $\overline{r}_1,\overline{r}_2,\ldots$ be the sequence of the closures of the border edges belonging to $\mathfrak{C}$, obtained following the clockwise orientation of the perimeter of $\cP$ starting from $\overline{r}_1$. For all $i\in\{2,\dots,n-1\}$ considering three consecutive cells $A_{i-1},A_{i}$ and $A_{i+1}$ of $\cP$, the possible arrangements of $\overline{r}_{j}$ and  $\overline{r}_{j+1}$ are displayed in Figure \ref*{fig:ispolygon}, up to just rotations. Then it is easy to see that $\overline{r}_{j}\cap \overline{r}_{j+1}$ is exactly a common endpoint of the two segments $\overline{r}_j$ and $\overline{r}_{j+1}$ for all $j$. Moreover, since $A_n$ and $A_1$ have an edge in common, there exists $m\in \NN$ such that $r_m$ is the border edge of $A_1$ where $r_m\cap r_1$ is the upper left corner of $A_1$, so	$\overline{r}_m\cap \overline{r}_1$ is a common endpoint of $\overline{r}_m$ and $\overline{r}_1$.  
 	\begin{figure}[h]
 		\subfloat[][]{\includegraphics[scale=0.8]{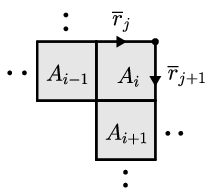}}\qquad  \qquad
 		\subfloat[][]{\includegraphics[scale=0.8]{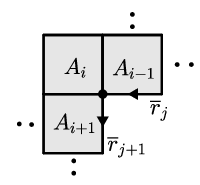}}\qquad \qquad
 		\subfloat[][]{\includegraphics[scale=0.8]{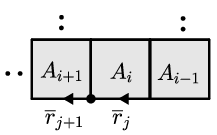}}
 		\caption{}
 		\label{fig:ispolygon}
 	\end{figure}
 	Therefore, $\mathfrak{C}$ is a rectilinear polygon. We prove that $\mathfrak{C}$ is simple. First of all, we recall that the clockwise orientation along $\cP$ induces an analogous orientation along the polygon $\mathfrak{C}$. By contradiction we suppose that $\mathfrak{C}$ is not simple, so there exists a self-intersection. Considering the orientation of $\mathfrak{C}$, we can distinguish exactly three cases up to just rotations, described in Figure~\ref{fig:selfintersect}, where the lines $a,b$ and $c$ belong to $\mathfrak{C}$. In the first case in Figure~\ref{fig:selfintersect} (A) we obtain that the common edge of $A_i$ and $A_{i+1}$ belongs to $\mathfrak{C}$, but this is a contradiction since $\mathfrak{C}$ contains only border edges. The same contradiction rises in the second and third case, considering respectively the common edge of $A_{i+1}$ and $A_j$ as in Figure~\ref{fig:selfintersect} (B), and the common edge of $A_{i}$ and $A_j$ as in Figure~\ref{fig:selfintersect} (C). Therefore $\mathfrak{C}$ is a simple rectilinear polygon; for instance, see Figure \ref{dim05} (B).
 	\begin{figure}[h]
 		\subfloat[][]{\includegraphics[scale=0.8]{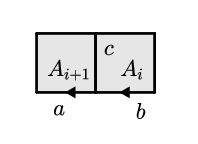}}\qquad  \qquad
 		\subfloat[][]{\includegraphics[scale=0.8]{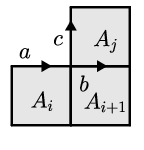}}\qquad \qquad
 		\subfloat[][]{\includegraphics[scale=0.8]{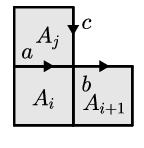}}
 		\caption{}
 		\label{fig:selfintersect}
 	\end{figure}
 	\noindent 
  In general it is easy to see geometrically that, walking clockwise along the perimeter of a rectilinear simple polygon, the interior of the polygon is on the right of the perimeter. Hence the cells of $\cP$ are all situated in the interior of $\mathfrak{C}$. By Lemma~\ref{A closed path contains a block} we can consider a part of $\cP$ arranged as in Figure~\ref{una3} (A), up to just rotations. By Remark \ref{simple-remark} (1) we have that $C$ and $D$ do not belong to $\cP$. We prove that $C$ and $D$ are neither both internal or both external to the polygon bounded by $\mathfrak{C}$. We denote by $r_C$ and $r_D$ the edges respectively of $C$ and $D$ that are border edges of $A_i$. We observe that either $\overline{r}_{C}\in \mathfrak{C}$ or $\overline{r}_{D}\in \mathfrak{C}$. We may assume that $\overline{r}_{C}\in \mathfrak{C}$, so $\overline{r}_{C}$ belongs to an edge of $\mathfrak{C}$, whose orientation goes from South to North, with reference to Figure \ref{una3} (B). 
 	In such a case $C$ is external to the polygon bounded by $\mathfrak{C}$. We prove that $D$ is in the interior of $\mathfrak{C}$. Suppose by contradiction that $D$ is external to the polygon bounded by $\mathfrak{C}$, so $D$ is on the left of $\mathfrak{C}$ with respect to its orientation.
 	In such a case, the only possibility is that $A_i$ is on the right of $\mathfrak{C}$ with respect the orientation of $\mathfrak{C}$ and $\overline{r}_D\in\mathfrak{C}$. Therefore $\overline{r}_{D}$ belongs to another edge of $\mathfrak{C}$, whose orientation goes from North to South. Let $A$ be the cell at North of $A_i$. The situation described above is summarized in Figure \ref{una3} (B).
 	\begin{figure}[h]
 		\subfloat[][]{\includegraphics[scale=0.7]{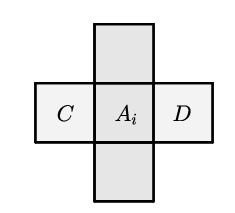}}\qquad  \qquad
 		\subfloat[][]{\includegraphics[scale=0.7]{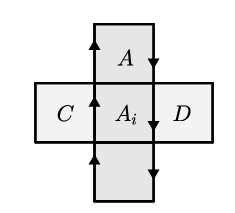}}
 		\caption{}
 		\label{una3}
 	\end{figure}	 
 	\noindent Walking along the edge of $\mathfrak{C}$ containing $\overline{r}_{C}$ we have that $A=A_{i+1}$. Walking along the edge of $\mathfrak{C}$ containing $\overline{r}_{D}$ we have that $A=A_{i-1}$. Then we have $A_{i+1}=A_{i-1}$, that is a contradiction with (3) of Definition \ref{def closed path}. By similar arguments, if we assume that $\overline{r}_D\in\mathfrak{C}$ then $C$ and $D$ are respectively internal and external to $\mathfrak{C}$. We assume without loss of generality that $C$ is internal to $\mathfrak{C}$ and $D$ is external to $\mathfrak{C}$. \\
 	Suppose that $\cP$ is a simple polyomino. Then there exists a path $\cF:F_1,F_2,\dots,F_t$, which connects $C$ and $D$, and $F_k$ does not belong to $\cP$ for all $k\in \{1,\dots,t\}$. Since $C$ is internal to $\mathfrak{C}$ and $D$ is external to $\mathfrak{C}$, there exist $k\in\{1,\dots,t-1\}$ and a border edge $r$ of a cell $F$ of $\cP$ such that $E(F_k)\cap E(F_{k+1})=\{r\}$. We observe that $F$, $F_k$ and $F_{k+1}$ are three cells of $\Z^2$ such that they have the edge $r$ in common and  $F_k\neq F_{k+1}$ because $F_k$,$F_{k+1}$ belong to $\cF$. Then either $F=F_k$ or $F=F_{k+1}$. But it is a contradiction because $F\in \cP$ and $F\notin \cP$ at the same time. Therefore $\cP$ is a non-simple polyomino. \\
 	2) Suppose that $\cP$ has more than one hole. In particular we can assume that $\cP$ has two holes $\cH_1$ and $\cH_2$. Then there exist three cells $B_1,C_1,D_1$ of $\mathbb{Z}^2$ such that $B_1 \in \cH_1$, $C_1 \in \cH_2$ and $D_1$ is in the exterior of $\cP$.
 In particular there does not exist any path of cells not belonging to $\cP$ and linking $B_1$ to $C_1$, $B_1$ to $D_1$  and $C_1$ to $D_1$. By Lemma~\ref{A closed path contains a block} we can consider a part of $\cP$ arranged as in Figure~\ref{fig:simple}(B). Considering the cells $B_1$ and $A_i$, we have by Remark~\ref{simple-remark} (2) that there exists a path $\mathcal{C}_{1}:B_1,\dots,B_m$ of cells not in $\cP$ such that $B_m\cap A_i$ is an edge of $B_m$ and $A_i$. The same holds for $C_1,A_i$ and $D_1,A_i$, hence there exist two paths $\mathcal{C}_{2}:C_1,\dots,C_n$ and $\mathcal{C}_{3}:D_1,\dots,D_r$ of cells not in $\cP$ such that $C_n\cap A_i$ is an edge of $C_n$ and $A_i$ and $D_r\cap A_i$ is an edge of $D_r$ and $A_i$.
For the shape of this configuration then, among those paths, there are two, for instance $\mathcal{C}_{1}$ and $\mathcal{C}_2$, having $C$ or $D$ as the last cells. Let $\cC_2^{rev}$ be the path obtained by $\cC_2$ inverting the order of the cells, that is $\cC_2^{rev}:C_1',\dots,C_n'$ where $C_i'=C_{n-i+1}$ for all $i=1,\dots,n$. So, by Remark \ref{Remark: walk diventa path} we have that $\mathcal{C}_1\cup \mathcal{C}_2^{rev}$ contains a path of cells not belonging to $\cP$ linking $B_1$ to $C_1$, that is a contradiction.\\
3) Assume that $r=1$ and suppose that $A_i,A_{i+1}$ are arranged as in Figure~\ref{fig:simple}(A). We can suppose that $E$ belongs to the hole of $\cP$ and $D$ is in the exterior of $\cP$. Let $H_1,H_2$ be two cells not belonging to $\cP'$. Suppose that $H_1$ belongs to the hole of $\cP$ and $H_2$ is exterior to $\cP$, then there exist two paths $\mathcal{C}_1,\mathcal{C}_2$ of cells not belonging to $\cP$ (so, not belonging to $\cP'$) linking $H_1$ to $E$ and $D$ to $H_2$ respectively. We set $\cC':E,A_i,A_{i+1},D$. Therefore, by Remark \ref{Remark: walk diventa path} we have that $\mathcal{C}_{1}\cup \cC'\cup \mathcal{C}_{2}$ contains a path of cells not belonging to $\cP'$ linking $H_1$ to $H_2$. We obtain easily the same conclusion if both $H_1, H_2$ belong to the hole, or both $H_1, H_2$ are in the exterior of $\cP$ and if one between $H_1$ or $H_2$ is the cell $A_i$ or $A_{i+1}$. By similar arguments we obtain the same conclusion in case $A_i,A_{i+1}$ are arranged as in Figure~\ref{fig:simple}(B). So if $r=1$ then $\cP'$ is a simple polyomino. The case $r>1$ can be proved by similar arguments.
\end{proof}

\begin{defn} \rm
\noindent Let $\cP$ be a polyomino. A path of five cells $A_1, A_2, A_3, A_4, A_5$ of $\cP$ is called an \textit{L-configuration} if the two sequences $A_1, A_2, A_3$ and $A_3, A_4, A_5$ go in two orthogonal directions. 
\end{defn}

\begin{figure}[h]
\centering
\includegraphics[scale=0.7]{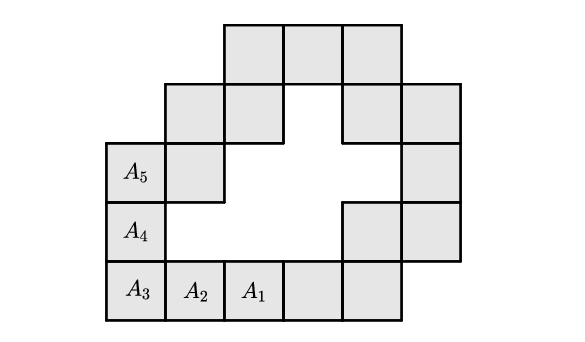}
	\caption{A closed path with an L-configuration.}
	\label{L-configuration}
\end{figure}

\begin{prop}\label{L-conf then no zig zag}
	Let $\cP$ be a closed path. If $\cP$ has at least an L-configuration, then $\cP$ contains no zig-zag walks.
\end{prop}
\begin{proof}
	We suppose that $\cP$ contains a zig-zag walk $\cW:I_1,\dots,I_\ell$. Let $A_1, A_2, A_3, A_4, A_5$ be an L-configuration. We denote by $a,b$ the diagonal corners of $A_3$ and by $c,d$ the anti-diagonal ones. We may suppose that $A_2\cap A_3=\{a,d\}$ and $A_3\cap A_4=\{d,b\}$, since similar arguments can be used in the other cases. Since $I_i\cap I_{i+1}\neq \emptyset$ for all $i\in \{1,\dots,\ell-1\}$, there exists $r\in \{1,\dots,\ell\}$ such that $A_1,A_2\in \cP(I_r)$ and $A_4,A_5\in \cP(I_s)$, where $s=r+1$ or $s=r-1$, with $I_0=I_{\ell}$ and $I_{\ell+1}=I_1$. We may suppose that $s=r+1$ (see Figure \ref{dim prop L-configuration}). 
	\begin{figure}[h]
		\centering
\includegraphics[scale=0.75]{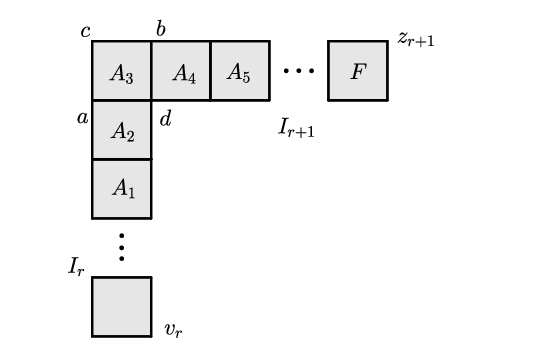}
		\caption{}
		\label{dim prop L-configuration}
	\end{figure}
	We prove that $v_{r+1}=d$. If $v_{r+1}\neq d$, then $\{v_{r+1},d\}\subseteq I_r\cap I_{r+1}$, that is a contradiction. Since $v_{r+1}=d$ and $\cP(I_r)\supseteq \{A_2\}$, the anti-diagonal corner $z_r$ of $I_r$ is equal to the vertex $a$ of $A_3$. Let $F$ be the cell of $\cP$ such that $\cP(I_{r+1})=[A_4,F]$. Then $[z_r,z_{r+1}]=V([A_3,F])$. $V([A_3,F])$ is an inner interval of $\cP$ such that $z_{r},z_{r+1}$ belong to it. This is a contradiction.     
	\end{proof}

\begin{rmk}\rm 	Notice that it is possible to build closed paths, which contain no L-configurations and no zig-zag walks; see Figure \ref{No L-conf no zig zag}.
	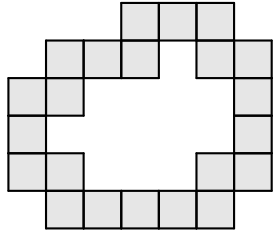
\begin{figure}[h]
		\centering
		\begin{tikzpicture}[line cap=round,line join=round,>=triangle 45,x=1.0cm,y=1.0cm]
		\clip(-6.,-4.) rectangle (-1.5,0.);
		\fill[line width=0.8pt,fill=black,fill opacity=0.10000000149011612] (-5.,-3.5) -- (-4.5,-3.5) -- (-4.5,-3.) -- (-5.,-3.) -- cycle;
		\fill[line width=0.8pt,fill=black,fill opacity=0.10000000149011612] (-5.,-3.) -- (-4.5,-3.) -- (-4.5,-2.5) -- (-5.,-2.5) -- cycle;
		\fill[line width=0.8pt,fill=black,fill opacity=0.10000000149011612] (-5.5,-3.) -- (-5.,-3.) -- (-5.,-2.5) -- (-5.5,-2.5) -- cycle;
		\fill[line width=0.8pt,fill=black,fill opacity=0.10000000149011612] (-5.5,-2.5) -- (-5.,-2.5) -- (-5.,-2.) -- (-5.5,-2.) -- cycle;
		\fill[line width=0.8pt,fill=black,fill opacity=0.10000000149011612] (-5.5,-2.) -- (-5.,-2.) -- (-5.,-1.5) -- (-5.5,-1.5) -- cycle;
		\fill[line width=0.8pt,fill=black,fill opacity=0.10000000149011612] (-4.5,-3.5) -- (-4.,-3.5) -- (-4.,-3.) -- (-4.5,-3.) -- cycle;
		\fill[line width=0.8pt,fill=black,fill opacity=0.10000000149011612] (-5.,-2.) -- (-4.5,-2.) -- (-4.5,-1.5) -- (-5.,-1.5) -- cycle;
		\fill[line width=0.8pt,fill=black,fill opacity=0.10000000149011612] (-5.,-1.5) -- (-4.5,-1.5) -- (-4.5,-1.) -- (-5.,-1.) -- cycle;
		\fill[line width=0.8pt,fill=black,fill opacity=0.10000000149011612] (-4.5,-1.5) -- (-4.,-1.5) -- (-4.,-1.) -- (-4.5,-1.) -- cycle;
		\fill[line width=0.8pt,fill=black,fill opacity=0.10000000149011612] (-4.,-1.5) -- (-3.5,-1.5) -- (-3.5,-1.) -- (-4.,-1.) -- cycle;
		\fill[line width=0.8pt,fill=black,fill opacity=0.10000000149011612] (-4.,-1.) -- (-3.5,-1.) -- (-3.5,-0.5) -- (-4.,-0.5) -- cycle;
		\fill[line width=0.8pt,fill=black,fill opacity=0.10000000149011612] (-3.5,-1.) -- (-3.,-1.) -- (-3.,-0.5) -- (-3.5,-0.5) -- cycle;
		\fill[line width=0.8pt,fill=black,fill opacity=0.10000000149011612] (-3.,-1.) -- (-2.5,-1.) -- (-2.5,-0.5) -- (-3.,-0.5) -- cycle;
		\fill[line width=0.8pt,fill=black,fill opacity=0.10000000149011612] (-3.,-1.5) -- (-2.5,-1.5) -- (-2.5,-1.) -- (-3.,-1.) -- cycle;
		\fill[line width=0.8pt,fill=black,fill opacity=0.10000000149011612] (-2.5,-1.5) -- (-2.,-1.5) -- (-2.,-1.) -- (-2.5,-1.) -- cycle;
		\fill[line width=0.8pt,fill=black,fill opacity=0.10000000149011612] (-2.5,-2.) -- (-2.,-2.) -- (-2.,-1.5) -- (-2.5,-1.5) -- cycle;
		\fill[line width=0.8pt,fill=black,fill opacity=0.10000000149011612] (-2.5,-2.5) -- (-2.,-2.5) -- (-2.,-2.) -- (-2.5,-2.) -- cycle;
		\fill[line width=0.8pt,fill=black,fill opacity=0.10000000149011612] (-2.5,-3.) -- (-2.,-3.) -- (-2.,-2.5) -- (-2.5,-2.5) -- cycle;
		\fill[line width=0.8pt,fill=black,fill opacity=0.10000000149011612] (-3.,-3.) -- (-2.5,-3.) -- (-2.5,-2.5) -- (-3.,-2.5) -- cycle;
		\fill[line width=0.8pt,fill=black,fill opacity=0.10000000149011612] (-3.,-3.5) -- (-2.5,-3.5) -- (-2.5,-3.) -- (-3.,-3.) -- cycle;
		\fill[line width=0.8pt,fill=black,fill opacity=0.10000000149011612] (-3.5,-3.5) -- (-3.,-3.5) -- (-3.,-3.) -- (-3.5,-3.) -- cycle;
		\fill[line width=0.8pt,fill=black,fill opacity=0.10000000149011612] (-4.,-3.5) -- (-3.5,-3.5) -- (-3.5,-3.) -- (-4.,-3.) -- cycle;
		\draw [line width=0.8pt] (-5.,-3.5)-- (-4.5,-3.5);
		\draw [line width=0.8pt] (-4.5,-3.5)-- (-4.5,-3.);
		\draw [line width=0.8pt] (-4.5,-3.)-- (-5.,-3.);
		\draw [line width=0.8pt] (-5.,-3.)-- (-5.,-3.5);
		\draw [line width=0.8pt] (-5.,-3.)-- (-4.5,-3.);
		\draw [line width=0.8pt] (-4.5,-3.)-- (-4.5,-2.5);
		\draw [line width=0.8pt] (-4.5,-2.5)-- (-5.,-2.5);
		\draw [line width=0.8pt] (-5.,-2.5)-- (-5.,-3.);
		\draw [line width=0.8pt] (-5.5,-3.)-- (-5.,-3.);
		\draw [line width=0.8pt] (-5.,-3.)-- (-5.,-2.5);
		\draw [line width=0.8pt] (-5.,-2.5)-- (-5.5,-2.5);
		\draw [line width=0.8pt] (-5.5,-2.5)-- (-5.5,-3.);
		\draw [line width=0.8pt] (-5.5,-2.5)-- (-5.,-2.5);
		\draw [line width=0.8pt] (-5.,-2.5)-- (-5.,-2.);
		\draw [line width=0.8pt] (-5.,-2.)-- (-5.5,-2.);
		\draw [line width=0.8pt] (-5.5,-2.)-- (-5.5,-2.5);
		\draw [line width=0.8pt] (-5.5,-2.)-- (-5.,-2.);
		\draw [line width=0.8pt] (-5.,-2.)-- (-5.,-1.5);
		\draw [line width=0.8pt] (-5.,-1.5)-- (-5.5,-1.5);
		\draw [line width=0.8pt] (-5.5,-1.5)-- (-5.5,-2.);
		\draw [line width=0.8pt] (-4.5,-3.5)-- (-4.,-3.5);
		\draw [line width=0.8pt] (-4.,-3.5)-- (-4.,-3.);
		\draw [line width=0.8pt] (-4.,-3.)-- (-4.5,-3.);
		\draw [line width=0.8pt] (-4.5,-3.)-- (-4.5,-3.5);
		\draw [line width=0.8pt] (-5.,-2.)-- (-4.5,-2.);
		\draw [line width=0.8pt] (-4.5,-2.)-- (-4.5,-1.5);
		\draw [line width=0.8pt] (-4.5,-1.5)-- (-5.,-1.5);
		\draw [line width=0.8pt] (-5.,-1.5)-- (-5.,-2.);
		\draw [line width=0.8pt] (-5.,-1.5)-- (-4.5,-1.5);
		\draw [line width=0.8pt] (-4.5,-1.5)-- (-4.5,-1.);
		\draw [line width=0.8pt] (-4.5,-1.)-- (-5.,-1.);
		\draw [line width=0.8pt] (-5.,-1.)-- (-5.,-1.5);
		\draw [line width=0.8pt] (-4.5,-1.5)-- (-4.,-1.5);
		\draw [line width=0.8pt] (-4.,-1.5)-- (-4.,-1.);
		\draw [line width=0.8pt] (-4.,-1.)-- (-4.5,-1.);
		\draw [line width=0.8pt] (-4.5,-1.)-- (-4.5,-1.5);
		\draw [line width=0.8pt] (-4.,-1.5)-- (-3.5,-1.5);
		\draw [line width=0.8pt] (-3.5,-1.5)-- (-3.5,-1.);
		\draw [line width=0.8pt] (-3.5,-1.)-- (-4.,-1.);
		\draw [line width=0.8pt] (-4.,-1.)-- (-4.,-1.5);
		\draw [line width=0.8pt] (-4.,-1.)-- (-3.5,-1.);
		\draw [line width=0.8pt] (-3.5,-1.)-- (-3.5,-0.5);
		\draw [line width=0.8pt] (-3.5,-0.5)-- (-4.,-0.5);
		\draw [line width=0.8pt] (-4.,-0.5)-- (-4.,-1.);
		\draw [line width=0.8pt] (-3.5,-1.)-- (-3.,-1.);
		\draw [line width=0.8pt] (-3.,-1.)-- (-3.,-0.5);
		\draw [line width=0.8pt] (-3.,-0.5)-- (-3.5,-0.5);
		\draw [line width=0.8pt] (-3.5,-0.5)-- (-3.5,-1.);
		\draw [line width=0.8pt] (-3.,-1.)-- (-2.5,-1.);
		\draw [line width=0.8pt] (-2.5,-1.)-- (-2.5,-0.5);
		\draw [line width=0.8pt] (-2.5,-0.5)-- (-3.,-0.5);
		\draw [line width=0.8pt] (-3.,-0.5)-- (-3.,-1.);
		\draw [line width=0.8pt] (-3.,-1.5)-- (-2.5,-1.5);
		\draw [line width=0.8pt] (-2.5,-1.5)-- (-2.5,-1.);
		\draw [line width=0.8pt] (-2.5,-1.)-- (-3.,-1.);
		\draw [line width=0.8pt] (-3.,-1.)-- (-3.,-1.5);
		\draw [line width=0.8pt] (-2.5,-1.5)-- (-2.,-1.5);
		\draw [line width=0.8pt] (-2.,-1.5)-- (-2.,-1.);
		\draw [line width=0.8pt] (-2.,-1.)-- (-2.5,-1.);
		\draw [line width=0.8pt] (-2.5,-1.)-- (-2.5,-1.5);
		\draw [line width=0.8pt] (-2.5,-2.)-- (-2.,-2.);
		\draw [line width=0.8pt] (-2.,-2.)-- (-2.,-1.5);
		\draw [line width=0.8pt] (-2.,-1.5)-- (-2.5,-1.5);
		\draw [line width=0.8pt] (-2.5,-1.5)-- (-2.5,-2.);
		\draw [line width=0.8pt] (-2.5,-2.5)-- (-2.,-2.5);
		\draw [line width=0.8pt] (-2.,-2.5)-- (-2.,-2.);
		\draw [line width=0.8pt] (-2.,-2.)-- (-2.5,-2.);
		\draw [line width=0.8pt] (-2.5,-2.)-- (-2.5,-2.5);
		\draw [line width=0.8pt] (-2.5,-3.)-- (-2.,-3.);
		\draw [line width=0.8pt] (-2.,-3.)-- (-2.,-2.5);
		\draw [line width=0.8pt] (-2.,-2.5)-- (-2.5,-2.5);
		\draw [line width=0.8pt] (-2.5,-2.5)-- (-2.5,-3.);
		\draw [line width=0.8pt] (-3.,-3.)-- (-2.5,-3.);
		\draw [line width=0.8pt] (-2.5,-3.)-- (-2.5,-2.5);
		\draw [line width=0.8pt] (-2.5,-2.5)-- (-3.,-2.5);
		\draw [line width=0.8pt] (-3.,-2.5)-- (-3.,-3.);
		\draw [line width=0.8pt] (-3.,-3.5)-- (-2.5,-3.5);
		\draw [line width=0.8pt] (-2.5,-3.5)-- (-2.5,-3.);
		\draw [line width=0.8pt] (-2.5,-3.)-- (-3.,-3.);
		\draw [line width=0.8pt] (-3.,-3.)-- (-3.,-3.5);
		\draw [line width=0.8pt] (-3.5,-3.5)-- (-3.,-3.5);
		\draw [line width=0.8pt] (-3.,-3.5)-- (-3.,-3.);
		\draw [line width=0.8pt] (-3.,-3.)-- (-3.5,-3.);
		\draw [line width=0.8pt] (-3.5,-3.)-- (-3.5,-3.5);
		\draw [line width=0.8pt] (-4.,-3.5)-- (-3.5,-3.5);
		\draw [line width=0.8pt] (-3.5,-3.5)-- (-3.5,-3.);
		\draw [line width=0.8pt] (-3.5,-3.)-- (-4.,-3.);
		\draw [line width=0.8pt] (-4.,-3.)-- (-4.,-3.5);
		\end{tikzpicture}
		\caption{A closed path without any $L$-configuration.}
		\label{No L-conf no zig zag}
	\end{figure}
\end{rmk}

\begin{rmk} \rm
If $\cP$ is a closed path and $\mathcal{B}_1,\mathcal{B}_{2}$ are two maximal horizontal (or vertical) blocks of $\cP$, then $|V(\cB_1)\cap V(\cB_{2})|=2$ or $V(\cB_1)\cap V(\cB_{2})=\emptyset$. If $V(\cB_1)\cap V(\cB_{2})=\{a,b\}$ then it is an edge belonging to $E(\mathcal{B}_1)\cap E(\mathcal{B}_{2})$. Observe also that $\cP$ is a union of blocks, not necessarily maximal, with the properties described above.
\end{rmk}

\begin{defn}\rm	Let $\cP$ be a polyomino. Let $\cB=\{\cB_i\}_{i=1,\dots,n}$ be a set of maximal horizontal (or vertical) blocks with length at least two, with $V(\cB_i)\cap V(\cB_{i+1})=\{a_i,b_i\}$, $a_i\neq b_i$ for all $i=1,\dots,n-1$. We say that $\cB$ is a \textit{ladder of $n$ steps} if $[a_i,b_i]$ is not on the same edge interval of $[a_{i+1},b_{i+1}]$ for all $i=1,\dots,n-2$. 
\end{defn}
	
	\begin{figure}[h]
		\centering
\includegraphics[scale=1]{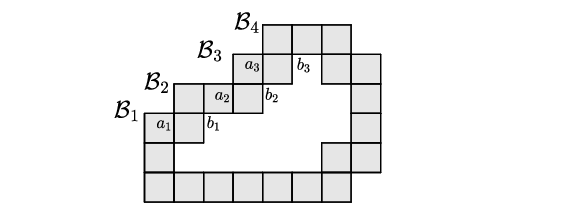}
		\caption{A closed path with a ladder of 4 steps.}
	\end{figure}
	\begin{prop}\label{ladder then no zig zag}
		Let $\cP$ be a closed path. If $\cP$ has a ladder of at least three steps, then $\cP$ contains no zig-zag walks.
	\end{prop}
	\begin{proof}
		Let $\cB=\{\cB_i\}_{i=1,\dots,n}$ be a ladder of $n$ steps. We may assume that $n=3$; for $n>3$ the arguments are similar.
		We can suppose $\cB_1,\cB_2,\cB_3$ are in horizontal position and the ladder is going up, otherwise we can reduce to this case by reflections or rotations (see Figure \ref{ladder dim}).
		\begin{figure}[h]
			\centering
\includegraphics[scale=0.7]{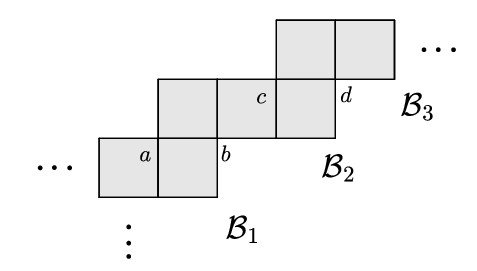}
			\caption{}
			\label{ladder dim}
		\end{figure}
		 Let $a,b,c,d$ be the vertices of $\cP$ such that $V(\cB_1)\cap V(\cB_2)=\{a,b\}$ and $V(\cB_2)\cap V(\cB_3)=\{c,d\}$. We assume that $\cP$ contains a zig-zag walk $\cW: I_1,\dots,I_\ell$. We suppose that there exists $i\in \{1,\dots,\ell\}$ such that $\cP(I_i)\subseteq \cB_1$, $\cP(I_{i+1})\subseteq \cB_2$ and $\cP(I_{i+2})\subseteq \cB_3$. One of the following cases can occur:
		\begin{enumerate}
			\item $I_{i}\cap I_{i+1}=\{a\}$ and  $I_{i+1}\cap I_{i+2}=\{c\}$;
			\item $I_{i}\cap I_{i+1}=\{a\}$ and  $I_{i+1}\cap I_{i+2}=\{d\}$;
			\item $I_{i}\cap I_{i+1}=\{b\}$ and  $I_{i+1}\cap I_{i+2}=\{c\}$;
			\item $I_{i}\cap I_{i+1}=\{b\}$ and  $I_{i+1}\cap I_{i+2}=\{d\}$.
		\end{enumerate}
		If the first one occurs, then $a,c$ should be on the same edge interval, a contradiction. The arguments are similar in the other cases.\\
		Let $A_1$ and $A_2$ be the cells, belonging respectively to $\cB_1$ and $\cB_2$, which have the edge $\{a,b\}$ in common. 
		Now we suppose that there exists $j\in\{1,\dots,\ell\}$ such that $\cP(I_j)$ contains $A_1$ and $A_2$, that is $I_j=V([A_1,A_2])$. Then there exists $r\in\{1,\dots,\ell\}$ such that $\cP(I_r)$ contains at least a cell in $\cB_{2}\cup \cB_3$, where $r=j-1$ or $r=j+1$, with $I_0=I_{\ell}$ and $I_{\ell+1}=I_1$. We may suppose that $r=j+1$. If $\mathcal{B}_2$ contains at least three cells then there does not exist any interval $I\subseteq V(\cB_{2})\cup V(\cB_3)$ such that $I_j\cap I$ is a vertex. In particular $|I_j\cap I_{j+1}|\neq 1$, that is a contradiction. If $\cB_2$ contains two cells then the only possibility to have $|I_j\cap I_{j+1}|=1$ is $v_{j+1}=c$. Moreover, in such a case, $v_j$ is the lower left corner of $A_1$; so $v_j$ and $v_{j+1}$ do not belong to the same edge interval, that is a contradiction to the definition of a zig-zag walk. 
		If there exists $j\in\{1,\dots,\ell\}$ such that $\cP(I_j)$ contains the cells $B_2$ of $\cB_2$ and $B_3$ of $\cB_3$, that have in common the edge $\{c,d\}$, similar arguments lead to a contradiction. 
	\end{proof}
	
	\begin{rmk}\rm 
			We note it is possible to build closed paths, which contain no ladders of $n \geq 2$ steps and no zig-zag walks; see Figure \ref*{rmk ladder configuration}.
			\begin{figure}[h]
			\includegraphics[scale=1]{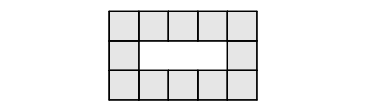}
				\caption{A closed path without any ladder.}
				\label{rmk ladder configuration}
			\end{figure}
	\end{rmk}

\section{Toric representation of closed paths with L-configurations}\label{Section L-conf toric}

	Let $\cP$ be a closed path with an L-configuration, consisting of the sequence of cells $A_1, A_2, A_3, A_4, A_5$. We denote by $a,b$ the diagonal corners of $A_3$ and by $c,d$ the anti-diagonal ones. We may suppose that $A_2\cap A_3=\{b,d\}$ and $A_3\cap A_4=\{c,b\}$, otherwise we can consider opportune reflections or rotations in order to have such an L-configuration. We also set $A_3=A$ (see Figure \ref{L conf toric}).
	\begin{figure}[h]
		\centering
\begin{tikzpicture}[line cap=round,line join=round,>=triangle 45,x=1.0cm,y=1.0cm]
\clip(-3.5,-2.5) rectangle (2.,2.);
\fill[line width=0.8pt,fill=black,fill opacity=0.10000000149011612] (-2.,-2.) -- (-1.,-2.) -- (-1.,-1.) -- (-2.,-1.) -- cycle;
\fill[line width=0.8pt,fill=black,fill opacity=0.10000000149011612] (-2.,-1.) -- (-1.,-1.) -- (-1.,0.) -- (-2.,0.) -- cycle;
\fill[line width=0.8pt,fill=black,fill opacity=0.10000000149011612] (-2.,0.) -- (-1.,0.) -- (-1.,1.) -- (-2.,1.) -- cycle;
\fill[line width=0.8pt,fill=black,fill opacity=0.10000000149011612] (-1.,-2.) -- (0.,-2.) -- (0.,-1.) -- (-1.,-1.) -- cycle;
\fill[line width=0.8pt,fill=black,fill opacity=0.10000000149011612] (0.,-2.) -- (1.,-2.) -- (1.,-1.) -- (0.,-1.) -- cycle;
\draw [line width=0.8pt] (-2.,-2.)-- (-1.,-2.);
\draw [line width=0.8pt] (-1.,-2.)-- (-1.,-1.);
\draw [line width=0.8pt] (-1.,-1.)-- (-2.,-1.);
\draw [line width=0.8pt] (-2.,-1.)-- (-2.,-2.);
\draw [line width=0.8pt] (-2.,-1.)-- (-1.,-1.);
\draw [line width=0.8pt] (-1.,-1.)-- (-1.,0.);
\draw [line width=0.8pt] (-1.,0.)-- (-2.,0.);
\draw [line width=0.8pt] (-2.,0.)-- (-2.,-1.);
\draw [line width=0.8pt] (-2.,0.)-- (-1.,0.);
\draw [line width=0.8pt] (-1.,0.)-- (-1.,1.);
\draw [line width=0.8pt] (-1.,1.)-- (-2.,1.);
\draw [line width=0.8pt] (-2.,1.)-- (-2.,0.);
\draw [line width=0.8pt] (-1.,-2.)-- (0.,-2.);
\draw [line width=0.8pt] (0.,-2.)-- (0.,-1.);
\draw [line width=0.8pt] (0.,-1.)-- (-1.,-1.);
\draw [line width=0.8pt] (-1.,-1.)-- (-1.,-2.);
\draw [line width=0.8pt] (0.,-2.)-- (1.,-2.);
\draw [line width=0.8pt] (1.,-2.)-- (1.,-1.);
\draw [line width=0.8pt] (1.,-1.)-- (0.,-1.);
\draw [line width=0.8pt] (0.,-1.)-- (0.,-2.);
\draw (0.16,-1.22) node[anchor=north west] {$A_1$};
\draw (-0.86,-1.2) node[anchor=north west] {$A_2$};
\draw (-3.54,-1.22) node[anchor=north west] {$A_3=A$};
\draw (-1.84,-0.24) node[anchor=north west] {$A_4$};
\draw (-1.84,0.72) node[anchor=north west] {$A_5$};
\draw (-2.46,-2.02) node[anchor=north west] {$a$};
\draw (-1.,-0.54) node[anchor=north west] {$b$};
\draw (-2.46,-0.7) node[anchor=north west] {$c$};
\draw (-1.06,-1.96) node[anchor=north west] {$d$};
\begin{scriptsize}
\draw [fill=black] (-2.4023839506451963,0.4173816473832763) circle (0.5pt);
\draw [fill=black] (-2.708480673626847,0.4173816473832763) circle (0.5pt);
\draw [fill=black] (-2.986750421791984,0.4173816473832763) circle (0.5pt);
\draw [fill=black] (1.312517187359383,-1.502679614956171) circle (0.5pt);
\draw [fill=black] (1.5629599607080065,-1.502679614956171) circle (0.5pt);
\draw [fill=black] (1.8134027340566297,-1.502679614956171) circle (0.5pt);
\draw [fill=black] (-1.5255743585989172,1.191801200177291) circle (0.5pt);
\draw [fill=black] (-1.5255743585989172,1.4521382850774889) circle (0.5pt);
\draw [fill=black] (-1.5255743585989172,1.7124753699776865) circle (0.5pt);
\end{scriptsize}
\end{tikzpicture}
		\caption{}
		\label{L conf toric}
	\end{figure}
	
	\noindent Let $\{V_i\}_{i\in I}$ be the sets of the maximal vertical edge intervals of $\cP$ and $\{H_j\}_{j\in J}$ be the set of the maximal horizontal edge intervals of $\cP$. Let $\{v_i\}_{i\in I}$ and $\{h_j\}_{j\in J}$ be the set of the variables associated respectively to $\{V_i\}_{i\in I}$ and $\{H_j\}_{j\in J}$. Let $w$ be another variable different from $v_i$ and $h_j$, $i\in I$ and $j\in J$.  
	  We define the following map:
\begin{align*}
	\alpha: V(\cP)&\longrightarrow K[\{v_i,h_j,w\}:i\in I,j\in J]\\
	r&\longmapsto  v_ih_jw^k
	\end{align*}
with $r\in V_i\cap H_j$, $k=0$ if $r\notin V(A)$, and $k=1$, if $r\in V(A)$. \\
The toric ring, denoted by $T_{\cP}$, is $K[\alpha(v):v\in V(\cP)]$. We denote by $S$ the polynomial ring $K[x_r:r\in V(\cP)]$ and we consider the following surjective ring homomorphism
	\begin{align*}
	\phi: S &\longrightarrow T_{\cP}\\
	\phi(x_r&)=\alpha(r)
	\end{align*}
The toric ideal $J_{\cP}$ is the kernel of $\phi$. 

\begin{prop}\label{closed path L then I in J}
	Let $\cP$ be a closed path with an L-configuration. Then $I_{\cP}\subseteq J_{\cP}$.
\end{prop}
\begin{proof}
	Let $f$ be a binomial that is a generator of $I_{\cP}$. Then there exists an inner interval $[p,q]$ of $\cP$, such that $f=x_px_q-x_rx_s$, where $r,s$ are the anti-diagonal corners of $[p,q]$. We prove that $f\in J_{\cP}$. Since $[p,q]$ is an inner interval, the vertices $p$, $r$ and $q$, $s$ are respectively on the same maximal vertical edge intervals and, similarly, the vertices $p$, $s$ and $q$, $r$ are respectively on the same maximal horizontal edge intervals. If $[p,q]\cap A = \emptyset$, then it is clear that $f\in J_{\cP}$. If $[p,q]=A$, then $p,q$ are the diagonal corners of $A$ and $r,s$ are the anti-diagonal ones, so $f\in J_{\cP}$. If $[p,q]\cap A\neq \emptyset$ and $[p,q]\neq A$, then a corner of $[p,q]$ belongs to $A$ and another one is not in $A$. We may assume that $p\in A$, in particular that $p=a$. Then $q\notin V(A)$, otherwise $[p,q]=A.$ Since $r$ and $s$ are the anti-diagonal corners of $[p,q]$, then $r=c$ and $s\notin A$. It follows that $f\in J_{\cP}$. Similar arguments hold in the other cases. 
\end{proof}

 \noindent By Proposition \ref{closed path L then I in J} and the definition of $\phi:S\rightarrow T_\cP$, we can use Lemma~\ref{Lemma shikama1} in the next Theorem, considering $J=J_\cP$.

\begin{thm}\label{Path with L conf is toric}\label{proof1}
	Let $\cP$ be a closed path with an L-configuration. Then $I_{\cP}= J_{\cP}$.	
\end{thm}
\begin{proof} By Proposition \ref{closed path L then I in J} we have $I_{\cP}\subseteq J_{\cP}$. We prove that $J_{\cP}\subseteq I_{\cP}$, showing the following two facts:
	\begin{enumerate}
		\item every binomial of degree two in $J_{\cP}$ belongs to $I_{\cP}$;
		\item every irredundant binomial in $J_{\cP}$ is of degree two.
	\end{enumerate}
We prove (1). Let $f=x_px_q-x_rx_s$ be a binomial in $J_{\cP}$. Without loss of generality we can assume that $p,q$ are the diagonal corners of the interval $[p,q]$. We denote by $v_p,h_p$ and $v_q,h_q$ the variables associated to the maximal horizontal and vertical edge intervals, which contain respectively $p$ and $q$. Consider that $\phi(x_px_q)=w^kv_ph_pv_qh_q=\phi(x_rx_s)$ with $k\in \{0,1,2\}$. The only possibility is that $r,s$ are the anti-diagonal corners of $[p,q]$ and that $[p, r]$, $[p, s]$, $[s, q]$ and $[r, q]$ are edge intervals of $\cP$. By contradiction, we assume that $[p, q]$ is not an inner interval of $\cP$, in particular there exists a
set of cells of $[p, q]$ that do not belong to $\cP$. Since $[p, r]$, $[p, s]$, $[s, q]$ and $[r, q]$
are edge intervals in $\cP$, then $[p,q]$ contains the hole $\mathcal{H}$ of $\cP$. In this case, the only possible arrangement of the cells of $\cP$ consists in having at least one of the corners $p,q,r$ and $s$ in $A$. We may assume that $p\in A$. Then $w$ divides $\phi(x_p)\phi(x_q)$ and so $w$ divides $\phi(x_r)$ or $\phi(x_s)$. From $w|\phi(x_r)$ (resp. $w|\phi(x_s)$) it follows that $r\in A$ (resp. $s\in A$). Since $\cH\subseteq [p,q]$, we have either $r$ or $s$ does not belong to $A$, so it is a contradiction. Hence $[p, q]$ is an inner interval of $\cP$.\\
\noindent We prove (2). We suppose that there exists a binomial $f$ in $J_{\cP}$ with $\deg f\geq3$, such that $f$ is irredundant. 
We suppose that every variable of $f$ is in $\{x_a:a\in V(\cP) \backslash V(A)\}$. We denote by $\cP'$ the simple polyomino obtained by removing the cells having vertices in common with $A$. We define the map $\phi'$ as the restriction of $\phi$ on $K[x_a:a\in V(\cP)\backslash V(A)]$ and we denote by $J_{\cP'}$ the kernel of $\phi'$. By Theorem 2.2 in \cite{Simple are prime}, we have that $J_{\cP'}=I_{\cP'}$, where $I_{\cP'}$ is the polyomino ideal associated to $\cP'$. We observe that $f$ is a binomial in $J_{\cP'}$. Since $J_{\cP'}\subset J_{\cP}$ and $f$ is irredundant in $J_{\cP}$, then $f$ is irredundant in $J_{\cP'}$. Then $f$ is irredundant in $I_{\cP'}$, that is a contradiction. It follows that there exists at least one variable in $f$, that corresponds to a vertex of $A$. We recall that $f=f^+-f^-\in J_{\cP}$, so $\phi(f^+)=\phi(f^-)$. We may suppose that there exists $v_1\in A$, such that $x_{v_1}$ divides $f^+$, that is $v_1\in V^+_f$. Then $w$ divides $\phi(f^+)=\phi(f^-)$, so there exists $v_1'\in A$, such that $x_{v_1'}$ divides $f^-$, that is $v_1'\in V^-_f$. If $v_1=v_1'$, then $f=x_{v_1}(\tilde{f}^+-\tilde{f}^-)$, where $\tilde{f}^+-\tilde{f}^-\in J_{\cP}$ and $\deg(\tilde{f}^+-\tilde{f}^-)<\deg f$, a contradiction. Then $v_1\neq v_1'$. Let $V_{v_1}$ and $H_{v_1}$ be the maximal vertical and horizontal edge intervals of $\cP$, which contain $v_1$. Then $v_{v_1}$ divides $\phi(f^+)=\phi(f^-)$, so there exists $v_2'\in V_{v_1}$ such that $x_{v_2'}$ divides $f^-$. Moreover $h_{v_1}$ divides $\phi(f^+)=\phi(f^-)$, so there exists $v_3'\in H_{v_1}$ such that $x_{v_3'}$ divides $f^-$. Let $V_{v_1'}$ and $H_{v_1'}$ be the maximal vertical and horizontal edge intervals of $\cP$, which contain $v_1'$. Then $v_{v_1'}$ divides $\phi(f^-)=\phi(f^+)$, so there exists $v_2\in V_{v_1'}$ such that $x_{v_2}$ divides $f^+$. Moreover $h_{v_1'}$ divides $\phi(f^-)=\phi(f^+)$, so there exists $v_3\in H_{v_1'}$ such that $x_{v_3}$ divides $f^+$. The following cases could occur:
\begin{enumerate}[(I)]
	\item $v_1$ and $v_1'$ are on the same vertical edge interval of $\cP$.
	\begin{figure}[h]
		\centering
	\begin{tikzpicture}[line cap=round,line join=round,>=triangle 45,x=1.0cm,y=1.0cm]
	\clip(-3.,-2.05) rectangle (1.,1.1);
	\fill[line width=0.8pt,fill=black,fill opacity=0.10000000149011612] (-2.,-1.) -- (-1.5,-1.) -- (-1.5,-0.5) -- (-2.,-0.5) -- cycle;
	\fill[line width=0.8pt,fill=black,fill opacity=0.25] (-2.,-1.5) -- (-1.5,-1.5) -- (-1.5,-1.) -- (-2.,-1.) -- cycle;
	\fill[line width=0.8pt,fill=black,fill opacity=0.10000000149011612] (-2.,-0.5) -- (-1.5,-0.5) -- (-1.5,0.) -- (-2.,0.) -- cycle;
	\fill[line width=0.8pt,fill=black,fill opacity=0.25] (-1.5,-1.5) -- (-1.,-1.5) -- (-1.,-1.) -- (-1.5,-1.) -- cycle;
	\fill[line width=0.8pt,fill=black,fill opacity=0.25] (-1.,-1.) -- (-1.,-1.5) -- (-0.5,-1.5) -- (-0.5,-1.) -- cycle;
	\fill[line width=0.8pt,fill=black,fill opacity=0.10000000149011612] (-1.,-1.) -- (-0.5,-1.) -- (-0.5,-0.5) -- (-1.,-0.5) -- cycle;
	\fill[line width=0.8pt,fill=black,fill opacity=0.10000000149011612] (-2.,0.) -- (-2.5,0.) -- (-2.5,-0.5) -- (-2.,-0.5) -- cycle;
	\fill[line width=0.8pt,fill=black,fill opacity=0.10000000149011612] (-2.5,0.) -- (-2.,0.) -- (-2.,0.5) -- (-2.5,0.5) -- cycle;
	\fill[line width=0.8pt,fill=black,fill opacity=0.10000000149011612] (-0.5,-1.) -- (0.,-1.) -- (0.,-0.5) -- (-0.5,-0.5) -- cycle;
	\draw [line width=0.8pt] (-2.,-1.)-- (-1.5,-1.);
	\draw [line width=0.8pt] (-1.5,-1.)-- (-1.5,-0.5);
	\draw [line width=0.8pt] (-1.5,-0.5)-- (-2.,-0.5);
	\draw [line width=0.8pt] (-2.,-0.5)-- (-2.,-1.);
	\draw [line width=0.8pt] (-2.,-1.5)-- (-1.5,-1.5);
	\draw [line width=0.8pt] (-1.5,-1.5)-- (-1.5,-1.);
	\draw [line width=0.8pt] (-1.5,-1.)-- (-2.,-1.);
	\draw [line width=0.8pt] (-2.,-1.)-- (-2.,-1.5);
	\draw [line width=0.8pt] (-2.,-0.5)-- (-1.5,-0.5);
	\draw [line width=0.8pt] (-1.5,-0.5)-- (-1.5,0.);
	\draw [line width=0.8pt] (-1.5,0.)-- (-2.,0.);
	\draw [line width=0.8pt] (-2.,0.)-- (-2.,-0.5);
	\draw [line width=0.8pt] (-1.5,-1.5)-- (-1.,-1.5);
	\draw [line width=0.8pt] (-1.,-1.5)-- (-1.,-1.);
	\draw [line width=0.8pt] (-1.,-1.)-- (-1.5,-1.);
	\draw [line width=0.8pt] (-1.5,-1.)-- (-1.5,-1.5);
	\draw [line width=0.8pt] (-1.,-1.)-- (-1.,-1.5);
	\draw [line width=0.8pt] (-1.,-1.5)-- (-0.5,-1.5);
	\draw [line width=0.8pt] (-0.5,-1.5)-- (-0.5,-1.);
	\draw [line width=0.8pt] (-0.5,-1.)-- (-1.,-1.);
	\draw [line width=0.8pt] (-1.,-1.)-- (-0.5,-1.);
	\draw [line width=0.8pt] (-0.5,-1.)-- (-0.5,-0.5);
	\draw [line width=0.8pt] (-0.5,-0.5)-- (-1.,-0.5);
	\draw [line width=0.8pt] (-1.,-0.5)-- (-1.,-1.);
	\draw [line width=0.8pt] (-2.,0.)-- (-2.5,0.);
	\draw [line width=0.8pt] (-2.5,0.)-- (-2.5,-0.5);
	\draw [line width=0.8pt] (-2.5,-0.5)-- (-2.,-0.5);
	\draw [line width=0.8pt] (-2.,-0.5)-- (-2.,0.);
	\draw [line width=0.8pt] (-2.5,0.)-- (-2.,0.);
	\draw [line width=0.8pt] (-2.,0.)-- (-2.,0.5);
	\draw [line width=0.8pt] (-2.,0.5)-- (-2.5,0.5);
	\draw [line width=0.8pt] (-2.5,0.5)-- (-2.5,0.);
	\draw [line width=0.8pt] (-0.5,-1.)-- (0.,-1.);
	\draw [line width=0.8pt] (0.,-1.)-- (0.,-0.5);
	\draw [line width=0.8pt] (0.,-0.5)-- (-0.5,-0.5);
	\draw [line width=0.8pt] (-0.5,-0.5)-- (-0.5,-1.);
	\draw (-2.06,-0.96) node[anchor=north west] {$A$};
	\draw (-2.52,-1.46) node[anchor=north west] {$v_1$};
	\draw (-2.52,-0.72) node[anchor=north west] {$v_1'$};
	\draw (-0.64,-1.38) node[anchor=north west] {$v_3'$};
	\begin{scriptsize}
	\draw [fill=black] (-2.2520289416202446,0.6460024357852483) circle (0.5pt);
	\draw [fill=black] (-2.2520289416202446,0.781379131820666) circle (0.5pt);
	\draw [fill=black] (-2.2520289416202446,0.9167558278560837) circle (0.5pt);
	\draw [fill=black] (0.13060090860310616,-0.7619152029830959) circle (0.5pt);
	\draw [fill=black] (0.2862841090438365,-0.7619152029830956) circle (0.5pt);
	\draw [fill=black] (0.4419673094845668,-0.7619152029830956) circle (0.5pt);
	\end{scriptsize}
	\end{tikzpicture}
		\caption{}
		\label{v1 e v1' are on vertical}
	\end{figure} 
For the structure of $\cP$, either $v_3$ or $v_3'$ is a vertex which identifies an inner interval of $\cP$ along with $v_1$ and $v_1'$ (see Figure \ref*{v1 e v1' are on vertical}). From Lemma \ref*{Lemma shikama1} a contradiction follows.
	\item 
	$v_1$ and $v_1'$ are on the same horizontal interval of $\cP$.
	\begin{figure}[h]
	\centering
\begin{tikzpicture}[line cap=round,line join=round,>=triangle 45,x=1.0cm,y=1.0cm]
\clip(-3.,-2.05) rectangle (1.,1.1);
\fill[line width=0.8pt,fill=black,fill opacity=0.25] (-2.,-1.) -- (-1.5,-1.) -- (-1.5,-0.5) -- (-2.,-0.5) -- cycle;
\fill[line width=0.8pt,fill=black,fill opacity=0.25] (-2.,-1.5) -- (-1.5,-1.5) -- (-1.5,-1.) -- (-2.,-1.) -- cycle;
\fill[line width=0.8pt,fill=black,fill opacity=0.25] (-2.,-0.5) -- (-1.5,-0.5) -- (-1.5,0.) -- (-2.,0.) -- cycle;
\fill[line width=0.8pt,fill=black,fill opacity=0.10000000149011612] (-1.5,-1.5) -- (-1.,-1.5) -- (-1.,-1.) -- (-1.5,-1.) -- cycle;
\fill[line width=0.8pt,fill=black,fill opacity=0.10000000149011612] (-1.,-1.) -- (-1.,-1.5) -- (-0.5,-1.5) -- (-0.5,-1.) -- cycle;
\fill[line width=0.8pt,fill=black,fill opacity=0.10000000149011612] (-1.,-1.) -- (-0.5,-1.) -- (-0.5,-0.5) -- (-1.,-0.5) -- cycle;
\fill[line width=0.8pt,fill=black,fill opacity=0.10000000149011612] (-2.,0.) -- (-2.5,0.) -- (-2.5,-0.5) -- (-2.,-0.5) -- cycle;
\fill[line width=0.8pt,fill=black,fill opacity=0.10000000149011612] (-2.5,0.) -- (-2.,0.) -- (-2.,0.5) -- (-2.5,0.5) -- cycle;
\fill[line width=0.8pt,fill=black,fill opacity=0.10000000149011612] (-0.5,-1.) -- (0.,-1.) -- (0.,-0.5) -- (-0.5,-0.5) -- cycle;
\draw [line width=0.8pt] (-2.,-1.)-- (-1.5,-1.);
\draw [line width=0.8pt] (-1.5,-1.)-- (-1.5,-0.5);
\draw [line width=0.8pt] (-1.5,-0.5)-- (-2.,-0.5);
\draw [line width=0.8pt] (-2.,-0.5)-- (-2.,-1.);
\draw [line width=0.8pt] (-2.,-1.5)-- (-1.5,-1.5);
\draw [line width=0.8pt] (-1.5,-1.5)-- (-1.5,-1.);
\draw [line width=0.8pt] (-1.5,-1.)-- (-2.,-1.);
\draw [line width=0.8pt] (-2.,-1.)-- (-2.,-1.5);
\draw [line width=0.8pt] (-2.,-0.5)-- (-1.5,-0.5);
\draw [line width=0.8pt] (-1.5,-0.5)-- (-1.5,0.);
\draw [line width=0.8pt] (-1.5,0.)-- (-2.,0.);
\draw [line width=0.8pt] (-2.,0.)-- (-2.,-0.5);
\draw [line width=0.8pt] (-1.5,-1.5)-- (-1.,-1.5);
\draw [line width=0.8pt] (-1.,-1.5)-- (-1.,-1.);
\draw [line width=0.8pt] (-1.,-1.)-- (-1.5,-1.);
\draw [line width=0.8pt] (-1.5,-1.)-- (-1.5,-1.5);
\draw [line width=0.8pt] (-1.,-1.)-- (-1.,-1.5);
\draw [line width=0.8pt] (-1.,-1.5)-- (-0.5,-1.5);
\draw [line width=0.8pt] (-0.5,-1.5)-- (-0.5,-1.);
\draw [line width=0.8pt] (-0.5,-1.)-- (-1.,-1.);
\draw [line width=0.8pt] (-1.,-1.)-- (-0.5,-1.);
\draw [line width=0.8pt] (-0.5,-1.)-- (-0.5,-0.5);
\draw [line width=0.8pt] (-0.5,-0.5)-- (-1.,-0.5);
\draw [line width=0.8pt] (-1.,-0.5)-- (-1.,-1.);
\draw [line width=0.8pt] (-2.,0.)-- (-2.5,0.);
\draw [line width=0.8pt] (-2.5,0.)-- (-2.5,-0.5);
\draw [line width=0.8pt] (-2.5,-0.5)-- (-2.,-0.5);
\draw [line width=0.8pt] (-2.,-0.5)-- (-2.,0.);
\draw [line width=0.8pt] (-2.5,0.)-- (-2.,0.);
\draw [line width=0.8pt] (-2.,0.)-- (-2.,0.5);
\draw [line width=0.8pt] (-2.,0.5)-- (-2.5,0.5);
\draw [line width=0.8pt] (-2.5,0.5)-- (-2.5,0.);
\draw [line width=0.8pt] (-0.5,-1.)-- (0.,-1.);
\draw [line width=0.8pt] (0.,-1.)-- (0.,-0.5);
\draw [line width=0.8pt] (0.,-0.5)-- (-0.5,-0.5);
\draw [line width=0.8pt] (-0.5,-0.5)-- (-0.5,-1.);
\draw (-2.,-0.96) node[anchor=north west] {$A$};
\draw (-2.46,-1.64) node[anchor=north west] {$v_1$};
\draw (-1.7,-1.52) node[anchor=north west] {$v_1'$};
\draw (-1.72,0.46) node[anchor=north west] {$v_2$};
\begin{scriptsize}
\draw [fill=black] (-2.2520289416202446,0.6460024357852483) circle (0.5pt);
\draw [fill=black] (-2.2520289416202446,0.781379131820666) circle (0.5pt);
\draw [fill=black] (-2.2520289416202446,0.9167558278560837) circle (0.5pt);
\draw [fill=black] (0.13060090860310616,-0.7619152029830959) circle (0.5pt);
\draw [fill=black] (0.2862841090438365,-0.7619152029830956) circle (0.5pt);
\draw [fill=black] (0.4419673094845668,-0.7619152029830956) circle (0.5pt);
\end{scriptsize}
\end{tikzpicture}
	\caption{}
	\label{v1 and v1' are on horizontal}	
		\end{figure}
	For the structure of $\cP$, either $v_2$ or $v_2'$ is a vertex which identifies an inner interval of $\cP$ along with $v_1$ and $v_1'$ (see Figure \ref*{v1 and v1' are on horizontal}). As before, by Lemma \ref*{Lemma shikama1}, we have a contradiction.
	\item $v_1$ and $v_1'$ are the diagonal corners of $A$.  We may suppose that $v_1=a$ and $v_1'=b$. We prove that $v_3'$ cannot be an anti-diagonal corner of $A$. If $v_3'$ is an anti-diagonal corner of $A$, then $v_3'=d$. For the structure of $\cP$, either $v_2$ or $v_2'$ is a vertex which identifies an inner interval of $\cP$ respectively with $v_1$ or $v_3'$. If $[v_1,v_2]$ is an inner interval, then we have a contradiction, applying Lemma \ref*{Lemma shikama1} to $v_1, v_3',v_2$. If the interval with anti-diagonal corners $v_2',v_1'$ is an inner interval, then we have a contradiction, by Lemma \ref*{Lemma shikama1} applied to $v_2',v_3', v_1$. By similar arguments, $v_3$,  $v_2$ and $v_2'$ cannot be anti-diagonal vertices of $A$.\\ For the structure of $\cP$, either $v_3$ or $v_3'$ is a vertex which identifies an inner interval of $\cP$ respectively with $v_1$ or $v_1'.$ We assume that $[v_1,v_3]$ is an inner interval of $\cP$. We denote by $g,h$ the anti-diagonal corners of $[v_1,v_3]$. For the structure of $\cP$, either $v_2$ or $v_2'$ is such that the interval identified by $g,v_2$ or $v_1',v_2'$ is inner to $\cP$. We assume that $[g,v_2]$ is an inner interval of $\cP$ (see Figure \ref{v1 and v1' not equal}). 
	\begin{figure}[h]
		\centering
\begin{tikzpicture}[line cap=round,line join=round,>=triangle 45,x=1.0cm,y=1.0cm]
\clip(-3.5,-4.) rectangle (0.8,0.);
\fill[line width=0.8pt,fill=black,fill opacity=0.20000000298023224] (-2.5,-2.) -- (-2.,-2.) -- (-2.,-1.5) -- (-2.5,-1.5) -- cycle;
\fill[line width=0.8pt,fill=black,fill opacity=0.20000000298023224] (-2.5,-2.5) -- (-2.,-2.5) -- (-2.,-2.) -- (-2.5,-2.) -- cycle;
\fill[line width=0.8pt,fill=black,fill opacity=0.30000001192092896] (-2.5,-3.) -- (-2.,-3.) -- (-2.,-2.5) -- (-2.5,-2.5) -- cycle;
\fill[line width=0.8pt,fill=black,fill opacity=0.30000001192092896] (-2.,-2.5) -- (-2.,-3.) -- (-1.5,-3.) -- (-1.5,-2.5) -- cycle;
\fill[line width=0.8pt,fill=black,fill opacity=0.30000001192092896] (-1.5,-2.5) -- (-1.5,-3.) -- (-1.,-3.) -- (-1.,-2.5) -- cycle;
\fill[line width=0.8pt,fill=black,fill opacity=0.10000000149011612] (-3.,-2.) -- (-2.5,-2.) -- (-2.5,-1.5) -- (-3.,-1.5) -- cycle;
\fill[line width=0.8pt,fill=black,fill opacity=0.10000000149011612] (-3.,-1.5) -- (-2.5,-1.5) -- (-2.5,-1.) -- (-3.,-1.) -- cycle;
\fill[line width=0.8pt,fill=black,fill opacity=0.10000000149011612] (-1.5,-3.5) -- (-1.,-3.5) -- (-1.,-3.) -- (-1.5,-3.) -- cycle;
\fill[line width=0.8pt,fill=black,fill opacity=0.10000000149011612] (-1.,-3.) -- (-1.,-3.5) -- (-0.5,-3.5) -- (-0.5,-3.) -- cycle;
\draw [line width=0.8pt] (-2.5,-2.)-- (-2.,-2.);
\draw [line width=0.8pt] (-2.,-2.)-- (-2.,-1.5);
\draw [line width=0.8pt] (-2.,-1.5)-- (-2.5,-1.5);
\draw [line width=0.8pt] (-2.5,-1.5)-- (-2.5,-2.);
\draw [line width=0.8pt] (-2.5,-2.5)-- (-2.,-2.5);
\draw [line width=0.8pt] (-2.,-2.5)-- (-2.,-2.);
\draw [line width=0.8pt] (-2.,-2.)-- (-2.5,-2.);
\draw [line width=0.8pt] (-2.5,-2.)-- (-2.5,-2.5);
\draw [line width=0.8pt] (-2.5,-3.)-- (-2.,-3.);
\draw [line width=0.8pt] (-2.,-3.)-- (-2.,-2.5);
\draw [line width=0.8pt] (-2.,-2.5)-- (-2.5,-2.5);
\draw [line width=0.8pt] (-2.5,-2.5)-- (-2.5,-3.);
\draw [line width=0.8pt] (-2.,-2.5)-- (-2.,-3.);
\draw [line width=0.8pt] (-2.,-3.)-- (-1.5,-3.);
\draw [line width=0.8pt] (-1.5,-3.)-- (-1.5,-2.5);
\draw [line width=0.8pt] (-1.5,-2.5)-- (-2.,-2.5);
\draw [line width=0.8pt] (-1.5,-2.5)-- (-1.5,-3.);
\draw [line width=0.8pt] (-1.5,-3.)-- (-1.,-3.);
\draw [line width=0.8pt] (-1.,-3.)-- (-1.,-2.5);
\draw [line width=0.8pt] (-1.,-2.5)-- (-1.5,-2.5);
\draw [line width=0.8pt] (-3.,-2.)-- (-2.5,-2.);
\draw [line width=0.8pt] (-2.5,-2.)-- (-2.5,-1.5);
\draw [line width=0.8pt] (-2.5,-1.5)-- (-3.,-1.5);
\draw [line width=0.8pt] (-3.,-1.5)-- (-3.,-2.);
\draw [line width=0.8pt] (-3.,-1.5)-- (-2.5,-1.5);
\draw [line width=0.8pt] (-2.5,-1.5)-- (-2.5,-1.);
\draw [line width=0.8pt] (-2.5,-1.)-- (-3.,-1.);
\draw [line width=0.8pt] (-3.,-1.)-- (-3.,-1.5);
\draw [line width=0.8pt] (-1.5,-3.5)-- (-1.,-3.5);
\draw [line width=0.8pt] (-1.,-3.5)-- (-1.,-3.);
\draw [line width=0.8pt] (-1.,-3.)-- (-1.5,-3.);
\draw [line width=0.8pt] (-1.5,-3.)-- (-1.5,-3.5);
\draw [line width=0.8pt] (-1.,-3.)-- (-1.,-3.5);
\draw [line width=0.8pt] (-1.,-3.5)-- (-0.5,-3.5);
\draw [line width=0.8pt] (-0.5,-3.5)-- (-0.5,-3.);
\draw [line width=0.8pt] (-0.5,-3.)-- (-1.,-3.);
\draw (-3.08,-2.96) node[anchor=north west] {$v_1$};
\draw (-1.,-2.34) node[anchor=north west] {$v_3$};
\draw (-3.04,-2.3) node[anchor=north west] {$g$};
\draw (-0.98,-2.96) node[anchor=north west] {$h$};
\draw (-2.,-1.24) node[anchor=north west] {$v_2$};
\draw (-2.56,-2.54) node[anchor=north west] {$A$};
\draw (-2.,-1.9) node[anchor=north west] {$v_1'$};
\begin{scriptsize}
\draw [fill=black] (-2.7542109448210343,-0.7975441118452805) circle (0.5pt);
\draw [fill=black] (-2.758686826577932,-0.6050811962986826) circle (0.5pt);
\draw [fill=black] (-2.758686826577932,-0.3991906354813919) circle (0.5pt);
\draw [fill=black] (-0.24178582204544496,-3.260844104625359) circle (0.5pt);
\draw [fill=black] (0.07355733214520188,-3.260844104625359) circle (0.5pt);
\draw [fill=black] (0.35570647010525436,-3.255311768586927) circle (0.5pt);
\end{scriptsize}
\end{tikzpicture}
		\caption{}
		\label{v1 and v1' not equal}
	\end{figure} 

 \noindent Then:
 $$ f=f^+-f^-=\frac{f^+}{x_{v_1}x_{v_3}}(x_{v_1}x_{v_3}-x_gx_h)+\frac{f^+}{x_{v_1}x_{v_3}}x_gx_h-f^-.$$
 Since $[v_1,v_3]$ is an inner interval of $\cP$, then $x_{v_1}x_{v_3}-x_gx_h \in I_{\cP}\subseteq J_{\cP}$. We set $\tilde{f}=\frac{f^+}{x_{v_1}x_{v_3}}x_gx_h-f^-$, $f_1=\frac{f^+}{x_{v_1}x_{v_3}}x_gx_h$ and $f_2=f^-$, so $\tilde{f}=f_1-f_2$. We observe that $\tilde{f}\in J_{\cP}$, $x_{v_2}x_g$ divides $f_1$ and $x_{v_1'}$ divides $f_2$. Since $v_2,g\in V^+_{\tilde{f}}$ and $v_1'\in V^-_{\tilde{f}}$, from Lemma \ref{Lemma shikama1} it follows that $\tilde{f}$ is redundant in $J_{\cP}$. Then $f$ in redundant in $J_{\cP}$, that is a contradiction. By similar arguments we can have the same conclusion in the other cases.
 \item $v_1$ and $v_1'$ are anti-diagonal corners of $A$. By arguments as in the previous case, we deduce that this one is not possible.
\end{enumerate}
Then $f$ is a redundant binomial in $J_{\cP}$. In conclusion we have $J_{\cP}\subseteq I_{\cP}$, hence 
$J_{\cP} = I_{\cP}$.
\end{proof}

\begin{coro}\label{P closed path L cof then prime}
Let $\cP$ be a closed path with an L-configuration. Then $I_{\cP}$ is prime.
\end{coro}

\section{Toric representation of closed paths with ladders of at least three steps} \label{Section toric ladder}
	Let $\cB=\{\cB_i\}_{i=1,\dots,m}$ be a maximal ladder of $m$ steps, $m>2$. After some convenient reflections or rotations of $\cP$, we can suppose that $\cB_1,\dots,\cB_m$ are in horizontal position and the ladder is going down. We suppose that the block $\cB_{m-1}$ is made up of $n$ cells, which we denote $A_1,\dots, A_n$ from left to right. We also denote by $a_i$ the lower left corner of $A_i$, for all $i=1,\dots,n$. Let $A$ be the cell of $\cB_m$, having an edge in common with $A_n$. We denote by $a,b$ the diagonal corners of $A$ and by $d$ the other anti-diagonal corner (see Figure \ref*{ladder toric representation}).
\begin{figure}[h]
\centering
\includegraphics[scale=0.7]{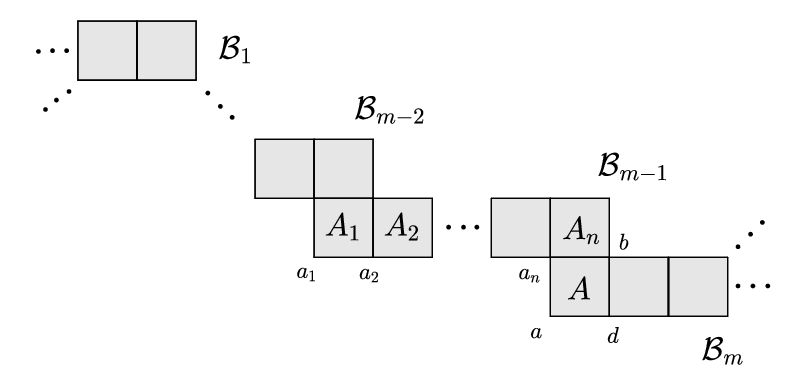}
\caption{}
\label{ladder toric representation}
\end{figure}

\noindent We also set $L_{\cB}=\{a_1,\dots,a_n,d,a,b \}$. As in the previous section, we denote by $\{V_i\}_{i\in I}$ the set of the maximal edge intervals of $\cP$ and by $\{H_j\}_{j\in J}$ the set of the maximal horizontal edge intervals of $\cP$. Let $\{v_i\}_{i\in I}$ and $\{h_j\}_{j\in J}$ be the sets of the variables associated respectively to $\{V_i\}_{i\in I}$ and $\{H_j\}_{j\in J}$. Let $\cH$ be the hole of $\cP$ and $w$ be another variable. We define the following map:

\begin{align*}
	\alpha: V(\cP)&\longrightarrow K[\{v_i,h_j,w\}:i\in I,j\in J]\\
	r&\longmapsto  v_ih_jw^k
	\end{align*}
with $V_i\cap H_j=\{r\}$ and where $k=0$, if $r\notin L_{\cB}$, and $k=1$, if $r\in L_{\cB}$. \\
We denote by $T_{\cP}$ the toric ring $K[\alpha(v):v\in V(\cP)]$ and by $J_{\cP}$ the kernel of the following surjective ring homomorphism:
\begin{align*}
\phi: S &\longrightarrow T_{\cP}\\
\phi(x_r&)=\alpha(r)
\end{align*}

\begin{prop}\label{I in J ladder conf}
	Let $\cP$ be a closed path with a ladder of $m$ steps ($m>2$). Then $I_{\cP}\subseteq J_{\cP}$.
\end{prop}
\begin{proof}
	Let $f$ be a binomial that is a generator of $I_{\cP}$. Then there exists an inner interval $[p,q]$ of $\cP$, such that $f=x_px_q-x_rx_s$, where $r,s$ are the anti-diagonal corners of $[p,q]$. If $[p,q]\cap L_{\cB}=\emptyset$, then $f\in J_{\cP}$. We suppose that  $[p,q]\cap L_{\cB}\neq\emptyset$. If $p,q\in L_{\cB}$, then $[p,q]=A$, so $f\in J_{\cP}$. If $p\in L_{\cB}$ and $q\notin L_{\cB}$, we have that either $r$ or $s$ belongs to $L_{\cB}$ for the structure of $\cP$, so $f\in J_{\cP}$. The case $p\notin L_{\cB}$ and $q\in L_{\cB}$ is not possible by construction. Then the desired conclusion follows.
\end{proof}

\noindent By Proposition \ref{I in J ladder conf} and the definition of $\phi:S\rightarrow T_\cP$, we can use Lemma~\ref{Lemma shikama1} in the next theorem, considering $J=J_\cP$.




\begin{thm}
		Let $\cP$ be a closed path with a ladder of $m$ steps ($m>2$). Then $I_{\cP}= J_{\cP}$.
		\label{proof2}
\end{thm}
\begin{proof}
Let $\cB=\{\cB_i\}_{i=1,\dots,m}$ be a maximal ladder of $m$ steps, $m>2$, where $\cB_1,\dots,\cB_m$ are in horizontal position and the ladder is going down. By Proposition \ref{I in J ladder conf}, we have $I_{\cP}\subseteq J_{\cP}$. Similar arguments as in (1) of Theorem \ref{Path with L conf is toric} allow us to prove that every binomial of degree two in $J_{\cP}$ belongs to $I_{\cP}$. We prove that every irredundant binomial in $J_{\cP}$ is of degree two. We suppose that there exists a binomial $f$ in $J_{\cP}$ with $\deg f\geq 3$, such that $f$ is irredundant. We prove that in $f$ there are not any variables associated to the vertices of $L_{\cB}$. 
 We suppose that there exists $v_1\in L_{\cB}$, such that $x_{v_1}$ divides $f^+$, that is $v_1\in V^+_f$. As in the proof of Theorem \ref{Path with L conf is toric}, we can find a vertex $v_1'\in L_{\cB}\cap V^-_{f}$, two vertices $v_2',v_3'\in V^-_{f}$ which are respectively on the same maximal vertical and horizontal edge intervals of $\cP$ containing $v_1$, and two vertices $v_2,v_3\in V^+_{f}$ which are  respectively on the same vertical and horizontal edge intervals of $\cP$ containing $v_1'$.
 The following cases could occur:
 \begin{enumerate}[(I)]
 	\item $v_1$ and $v_1'$ are on the same vertical edge interval of $\cP$. 
 	\begin{figure}[h]
 		\centering
 \begin{tikzpicture}[line cap=round,line join=round,>=triangle 45,x=1.0cm,y=1.0cm]
 \clip(-6.,-3.) rectangle (3.,1.);
 \fill[line width=0.8pt,fill=black,fill opacity=0.10000000149011612] (-4.,-1.) -- (-3.5,-1.) -- (-3.5,-0.5) -- (-4.,-0.5) -- cycle;
 \fill[line width=0.8pt,fill=black,fill opacity=0.10000000149011612] (-3.5,-1.) -- (-3.,-1.) -- (-3.,-0.5) -- (-3.5,-0.5) -- cycle;
 \fill[line width=0.8pt,fill=black,fill opacity=0.10000000149011612] (-3.,-1.) -- (-2.5,-1.) -- (-2.5,-0.5) -- (-3.,-0.5) -- cycle;
 \fill[line width=0.8pt,fill=black,fill opacity=0.10000000149011612] (-3.,-1.) -- (-3.,-1.5) -- (-2.5,-1.5) -- (-2.5,-1.) -- cycle;
 \fill[line width=0.8pt,fill=black,fill opacity=0.10000000149011612] (-2.5,-1.) -- (-2.5,-1.5) -- (-2.,-1.5) -- (-2.,-1.) -- cycle;
 \fill[line width=0.8pt,fill=black,fill opacity=0.10000000149011612] (-2.,-1.) -- (-2.,-1.5) -- (-1.5,-1.5) -- (-1.5,-1.) -- cycle;
 \fill[line width=0.8pt,fill=black,fill opacity=0.25] (-2.,-2.) -- (-1.5,-2.) -- (-1.5,-1.5) -- (-2.,-1.5) -- cycle;
 \fill[line width=0.8pt,fill=black,fill opacity=0.25] (-1.5,-2.) -- (-1.,-2.) -- (-1.,-1.5) -- (-1.5,-1.5) -- cycle;
 \fill[line width=0.8pt,fill=black,fill opacity=0.25] (-1.,-2.) -- (-0.5,-2.) -- (-0.5,-1.5) -- (-1.,-1.5) -- cycle;
 \fill[line width=0.8pt,fill=black,fill opacity=0.25] (-0.5,-2.) -- (0.,-2.) -- (0.,-1.5) -- (-0.5,-1.5) -- cycle;
 \fill[line width=0.8pt,fill=black,fill opacity=0.10000000149011612] (-0.5,-1.5) -- (0.,-1.5) -- (0.,-1.) -- (-0.5,-1.) -- cycle;
 \fill[line width=0.8pt,fill=black,fill opacity=0.10000000149011612] (0.,-1.5) -- (0.5,-1.5) -- (0.5,-1.) -- (0.,-1.) -- cycle;
 \fill[line width=0.8pt,fill=black,fill opacity=0.10000000149011612] (-4.5,0.) -- (-4.,0.) -- (-4.,0.5) -- (-4.5,0.5) -- cycle;
 \fill[line width=0.8pt,fill=black,fill opacity=0.10000000149011612] (-5.,0.) -- (-4.5,0.) -- (-4.5,0.5) -- (-5.,0.5) -- cycle;
 \draw [line width=0.8pt] (-4.,-1.)-- (-3.5,-1.);
 \draw [line width=0.8pt] (-3.5,-1.)-- (-3.5,-0.5);
 \draw [line width=0.8pt] (-3.5,-0.5)-- (-4.,-0.5);
 \draw [line width=0.8pt] (-4.,-0.5)-- (-4.,-1.);
 \draw [line width=0.8pt] (-3.5,-1.)-- (-3.,-1.);
 \draw [line width=0.8pt] (-3.,-1.)-- (-3.,-0.5);
 \draw [line width=0.8pt] (-3.,-0.5)-- (-3.5,-0.5);
 \draw [line width=0.8pt] (-3.5,-0.5)-- (-3.5,-1.);
 \draw [line width=0.8pt] (-3.,-1.)-- (-2.5,-1.);
 \draw [line width=0.8pt] (-2.5,-1.)-- (-2.5,-0.5);
 \draw [line width=0.8pt] (-2.5,-0.5)-- (-3.,-0.5);
 \draw [line width=0.8pt] (-3.,-0.5)-- (-3.,-1.);
 \draw [line width=0.8pt] (-3.,-1.)-- (-3.,-1.5);
 \draw [line width=0.8pt] (-3.,-1.5)-- (-2.5,-1.5);
 \draw [line width=0.8pt] (-2.5,-1.5)-- (-2.5,-1.);
 \draw [line width=0.8pt] (-2.5,-1.)-- (-3.,-1.);
 \draw [line width=0.8pt] (-2.5,-1.)-- (-2.5,-1.5);
 \draw [line width=0.8pt] (-2.5,-1.5)-- (-2.,-1.5);
 \draw [line width=0.8pt] (-2.,-1.5)-- (-2.,-1.);
 \draw [line width=0.8pt] (-2.,-1.)-- (-2.5,-1.);
 \draw [line width=0.8pt] (-2.,-1.)-- (-2.,-1.5);
 \draw [line width=0.8pt] (-2.,-1.5)-- (-1.5,-1.5);
 \draw [line width=0.8pt] (-1.5,-1.5)-- (-1.5,-1.);
 \draw [line width=0.8pt] (-1.5,-1.)-- (-2.,-1.);
 \draw [line width=0.8pt] (-2.,-2.)-- (-1.5,-2.);
 \draw [line width=0.8pt] (-1.5,-2.)-- (-1.5,-1.5);
 \draw [line width=0.8pt] (-1.5,-1.5)-- (-2.,-1.5);
 \draw [line width=0.8pt] (-2.,-1.5)-- (-2.,-2.);
 \draw [line width=0.8pt] (-1.5,-2.)-- (-1.,-2.);
 \draw [line width=0.8pt] (-1.,-2.)-- (-1.,-1.5);
 \draw [line width=0.8pt] (-1.,-1.5)-- (-1.5,-1.5);
 \draw [line width=0.8pt] (-1.5,-1.5)-- (-1.5,-2.);
 \draw [line width=0.8pt] (-1.,-2.)-- (-0.5,-2.);
 \draw [line width=0.8pt] (-0.5,-2.)-- (-0.5,-1.5);
 \draw [line width=0.8pt] (-0.5,-1.5)-- (-1.,-1.5);
 \draw [line width=0.8pt] (-1.,-1.5)-- (-1.,-2.);
 \draw [line width=0.8pt] (-0.5,-2.)-- (0.,-2.);
 \draw [line width=0.8pt] (0.,-2.)-- (0.,-1.5);
 \draw [line width=0.8pt] (0.,-1.5)-- (-0.5,-1.5);
 \draw [line width=0.8pt] (-0.5,-1.5)-- (-0.5,-2.);
 \draw [line width=0.8pt] (-0.5,-1.5)-- (0.,-1.5);
 \draw [line width=0.8pt] (0.,-1.5)-- (0.,-1.);
 \draw [line width=0.8pt] (0.,-1.)-- (-0.5,-1.);
 \draw [line width=0.8pt] (-0.5,-1.)-- (-0.5,-1.5);
 \draw [line width=0.8pt] (0.,-1.5)-- (0.5,-1.5);
 \draw [line width=0.8pt] (0.5,-1.5)-- (0.5,-1.);
 \draw [line width=0.8pt] (0.5,-1.)-- (0.,-1.);
 \draw [line width=0.8pt] (0.,-1.)-- (0.,-1.5);
 \draw (-3.98,0.56) node[anchor=north west] {$\mathcal{B}_1$};
 \draw [line width=0.8pt] (-4.5,0.)-- (-4.,0.);
 \draw [line width=0.8pt] (-4.,0.)-- (-4.,0.5);
 \draw [line width=0.8pt] (-4.,0.5)-- (-4.5,0.5);
 \draw [line width=0.8pt] (-4.5,0.5)-- (-4.5,0.);
 \draw [line width=0.8pt] (-5.,0.)-- (-4.5,0.);
 \draw [line width=0.8pt] (-4.5,0.)-- (-4.5,0.5);
 \draw [line width=0.8pt] (-4.5,0.5)-- (-5.,0.5);
 \draw [line width=0.8pt] (-5.,0.5)-- (-5.,0.);
 \draw (-3.58,0.1) node[anchor=north west] {$\mathcal{B}_{m-2}$};
 \draw (-2.42,-0.48) node[anchor=north west] {$\mathcal{B}_{m-1}$};
 \draw (-1.24,-0.94) node[anchor=north west] {$\mathcal{B}_m$};
 \draw (-2.06,-1.5) node[anchor=north west] {$A$};
 \draw (-2.5,-1.5) node[anchor=north west] {$v_1$};
 \draw (-2.5,-1.98) node[anchor=north west] {$v_1'$};
 \draw (-0.22,-2.12) node[anchor=north west] {$v_3$};
 \begin{scriptsize}
 \draw [fill=black] (-4.317902983809103,-0.1473382036869498) circle (0.5pt);
 \draw [fill=black] (-4.214255071138008,-0.2604086538735984) circle (0.5pt);
 \draw [fill=black] (-4.091762083435806,-0.373479104060247) circle (0.5pt);
 \draw [fill=black] (0.25202771123461104,-0.8163383672912873) circle (0.5pt);
 \draw [fill=black] (0.25202771123461104,-0.6561552295268684) circle (0.5pt);
 \draw [fill=black] (0.628929211856773,-1.2309300179756655) circle (0.5pt);
 \draw [fill=black] (0.7702672745900837,-1.2309300179756655) circle (0.5pt);
 \draw [fill=black] (0.9304504123545025,-1.2215074804601114) circle (0.5pt);
 \draw [fill=black] (0.25202771123461104,-0.47712701673134156) circle (0.5pt);
 \end{scriptsize}
 \end{tikzpicture}
 		\caption{}
 		\label{ladder v_1 e V_1' vertical}
 	\end{figure}
 For the structure of $\cP$ either $v_3$ or $v_3'$ is a vertex which identifies an inner interval of $\cP$ along with $v_1$ and $v_1'$ (see Figure \ref{ladder v_1 e V_1' vertical}). Lemma \ref{Lemma shikama1} leads to a contradiction.
 \item $v_1$ and $v_1'$ are on the same horizontal edge interval of $\cP$. If $\{v_1,v_1'\}=\{a,d\}$ or $\{v_1,v_1'\}=\{a_n,b\}$ or $\{v_1,v_1'\}\subseteq \{a_1,\dots,a_{n-1}\}$ with $n>2$, then either $v_2$ or $v_2'$ is a vertex which identifies an inner interval along with $v_1$ and $v_1'$. By using Lemma \ref{Lemma shikama1}, we have a contradiction. We suppose that $v_1\in \{a_1,\dots,a_{n-1}\}$ and $v_1'\in\{a_n,b\}$ or vice versa. We may assume that $v_1'=b$, because similar arguments hold when $v_1'=a_n$. If $v_2\notin L_{\cB}$, then we have a contradiction, using Lemma \ref{Lemma shikama1} to the vertices $v_1,v_1'$ and $v_2$. Let $v_2$ be in $L_{\cB}$; in particular the only possibility is $v_2=d$. Let $h_{v_2}$ be the variable associated with the horizontal interval of $v_2$. Then $h_{v_2}$ divides $\phi(f^+) =\phi(f^-)$, so we have two possibilities. The first one is $v_2\in V_f^-$, so $f=x_{v_{2}}(\tilde{f}^+-\tilde{f}^-)$, that is $f$ is not irredundant. Alternatively, there exists $\tilde{v}\in V_f^-$ such that $\tilde{v}$ is in the same horizontal edge interval of $v_2$; in particular $f$ is not irredundant by Lemma~\ref{Lemma shikama1} applied to the vertices $v_1',v_2,\tilde{v}$. In both cases we have a contradiction.
 \item $v_1$ and $v_1'$ are not on the same horizontal or vertical edge intervals of $\cP$. If they are diagonal or anti-diagonal vertices of $A$, then we have a contradiction, by similar arguments as in the last case (III) of Theorem \ref{Path with L conf is toric}. We suppose that $v_1\in\{a_1,\dots,a_{n-1}\}$ and $v_1' \in \{a,d\}$ (or vice versa). We may assume that $v_1'=d$, because similar arguments holds when $v_1'=a$. The vertex $v_2$ does not belong to $L_{\cB}$, otherwise we have a contradiction as in the previous case, so $[v_1,v_2]$ is an inner interval of $\cP$. We denote by $g,h$ the anti-diagonal vertices of $[v_1,v_2]$. We observe that $v_3\notin L_{\cB}$, otherwise we have a contradiction using the usual considerations to vertices $v_2,v_3,v_1'$. Then $h,v_3$ identify an inner interval of $\cP$, with $v_1'$ as diagonal corner (see Figure~\ref{v_1 and v_1' aren't on the same ,ladder} ).
 \begin{figure}[h]
 	\centering
\definecolor{uuuuuu}{rgb}{0.26666666666666666,0.26666666666666666,0.26666666666666666}
\begin{tikzpicture}[line cap=round,line join=round,>=triangle 45,x=1.0cm,y=1.0cm]
\clip(-6.,-3.) rectangle (3.,1.);
\fill[line width=0.8pt,fill=black,fill opacity=0.10000000149011612] (-4.,-1.) -- (-3.5,-1.) -- (-3.5,-0.5) -- (-4.,-0.5) -- cycle;
\fill[line width=0.8pt,fill=black,fill opacity=0.10000000149011612] (-3.5,-1.) -- (-3.,-1.) -- (-3.,-0.5) -- (-3.5,-0.5) -- cycle;
\fill[line width=0.8pt,fill=black,fill opacity=0.10000000149011612] (-3.,-1.) -- (-2.5,-1.) -- (-2.5,-0.5) -- (-3.,-0.5) -- cycle;
\fill[line width=0.8pt,fill=black,fill opacity=0.10000000149011612] (-3.,-1.) -- (-3.,-1.5) -- (-2.5,-1.5) -- (-2.5,-1.) -- cycle;
\fill[line width=0.8pt,fill=black,fill opacity=0.25] (-2.5,-1.) -- (-2.5,-1.5) -- (-2.,-1.5) -- (-2.,-1.) -- cycle;
\fill[line width=0.8pt,fill=black,fill opacity=0.25] (-2.,-1.) -- (-2.,-1.5) -- (-1.5,-1.5) -- (-1.5,-1.) -- cycle;
\fill[line width=0.8pt,fill=black,fill opacity=0.10000000149011612] (-2.,-2.) -- (-1.5,-2.) -- (-1.5,-1.5) -- (-2.,-1.5) -- cycle;
\fill[line width=0.8pt,fill=black,fill opacity=0.4000000059604645] (-1.5,-2.) -- (-1.,-2.) -- (-1.,-1.5) -- (-1.5,-1.5) -- cycle;
\fill[line width=0.8pt,fill=black,fill opacity=0.4000000059604645] (-1.,-2.) -- (-0.5,-2.) -- (-0.5,-1.5) -- (-1.,-1.5) -- cycle;
\fill[line width=0.8pt,fill=black,fill opacity=0.4000000059604645] (-0.5,-2.) -- (0.,-2.) -- (0.,-1.5) -- (-0.5,-1.5) -- cycle;
\fill[line width=0.8pt,fill=black,fill opacity=0.10000000149011612] (-0.5,-1.5) -- (0.,-1.5) -- (0.,-1.) -- (-0.5,-1.) -- cycle;
\fill[line width=0.8pt,fill=black,fill opacity=0.10000000149011612] (0.,-1.5) -- (0.5,-1.5) -- (0.5,-1.) -- (0.,-1.) -- cycle;
\fill[line width=0.8pt,fill=black,fill opacity=0.10000000149011612] (-4.5,0.) -- (-4.,0.) -- (-4.,0.5) -- (-4.5,0.5) -- cycle;
\fill[line width=0.8pt,fill=black,fill opacity=0.10000000149011612] (-5.,0.) -- (-4.5,0.) -- (-4.5,0.5) -- (-5.,0.5) -- cycle;
\draw [line width=0.8pt] (-4.,-1.)-- (-3.5,-1.);
\draw [line width=0.8pt] (-3.5,-1.)-- (-3.5,-0.5);
\draw [line width=0.8pt] (-3.5,-0.5)-- (-4.,-0.5);
\draw [line width=0.8pt] (-4.,-0.5)-- (-4.,-1.);
\draw [line width=0.8pt] (-3.5,-1.)-- (-3.,-1.);
\draw [line width=0.8pt] (-3.,-1.)-- (-3.,-0.5);
\draw [line width=0.8pt] (-3.,-0.5)-- (-3.5,-0.5);
\draw [line width=0.8pt] (-3.5,-0.5)-- (-3.5,-1.);
\draw [line width=0.8pt] (-3.,-1.)-- (-2.5,-1.);
\draw [line width=0.8pt] (-2.5,-1.)-- (-2.5,-0.5);
\draw [line width=0.8pt] (-2.5,-0.5)-- (-3.,-0.5);
\draw [line width=0.8pt] (-3.,-0.5)-- (-3.,-1.);
\draw [line width=0.8pt] (-3.,-1.)-- (-3.,-1.5);
\draw [line width=0.8pt] (-3.,-1.5)-- (-2.5,-1.5);
\draw [line width=0.8pt] (-2.5,-1.5)-- (-2.5,-1.);
\draw [line width=0.8pt] (-2.5,-1.)-- (-3.,-1.);
\draw [line width=0.8pt] (-2.5,-1.)-- (-2.5,-1.5);
\draw [line width=0.8pt] (-2.5,-1.5)-- (-2.,-1.5);
\draw [line width=0.8pt] (-2.,-1.5)-- (-2.,-1.);
\draw [line width=0.8pt] (-2.,-1.)-- (-2.5,-1.);
\draw [line width=0.8pt] (-2.,-1.)-- (-2.,-1.5);
\draw [line width=0.8pt] (-2.,-1.5)-- (-1.5,-1.5);
\draw [line width=0.8pt] (-1.5,-1.5)-- (-1.5,-1.);
\draw [line width=0.8pt] (-1.5,-1.)-- (-2.,-1.);
\draw [line width=0.8pt] (-2.,-2.)-- (-1.5,-2.);
\draw [line width=0.8pt] (-1.5,-2.)-- (-1.5,-1.5);
\draw [line width=0.8pt] (-1.5,-1.5)-- (-2.,-1.5);
\draw [line width=0.8pt] (-2.,-1.5)-- (-2.,-2.);
\draw [line width=0.8pt] (-1.5,-2.)-- (-1.,-2.);
\draw [line width=0.8pt] (-1.,-2.)-- (-1.,-1.5);
\draw [line width=0.8pt] (-1.,-1.5)-- (-1.5,-1.5);
\draw [line width=0.8pt] (-1.5,-1.5)-- (-1.5,-2.);
\draw [line width=0.8pt] (-1.,-2.)-- (-0.5,-2.);
\draw [line width=0.8pt] (-0.5,-2.)-- (-0.5,-1.5);
\draw [line width=0.8pt] (-0.5,-1.5)-- (-1.,-1.5);
\draw [line width=0.8pt] (-1.,-1.5)-- (-1.,-2.);
\draw [line width=0.8pt] (-0.5,-2.)-- (0.,-2.);
\draw [line width=0.8pt] (0.,-2.)-- (0.,-1.5);
\draw [line width=0.8pt] (0.,-1.5)-- (-0.5,-1.5);
\draw [line width=0.8pt] (-0.5,-1.5)-- (-0.5,-2.);
\draw [line width=0.8pt] (-0.5,-1.5)-- (0.,-1.5);
\draw [line width=0.8pt] (0.,-1.5)-- (0.,-1.);
\draw [line width=0.8pt] (0.,-1.)-- (-0.5,-1.);
\draw [line width=0.8pt] (-0.5,-1.)-- (-0.5,-1.5);
\draw [line width=0.8pt] (0.,-1.5)-- (0.5,-1.5);
\draw [line width=0.8pt] (0.5,-1.5)-- (0.5,-1.);
\draw [line width=0.8pt] (0.5,-1.)-- (0.,-1.);
\draw [line width=0.8pt] (0.,-1.)-- (0.,-1.5);
\draw (-3.98162701774016,0.4934041688675306) node[anchor=north west] {$\mathcal{B}_1$};
\draw [line width=0.8pt] (-4.5,0.)-- (-4.,0.);
\draw [line width=0.8pt] (-4.,0.)-- (-4.,0.5);
\draw [line width=0.8pt] (-4.,0.5)-- (-4.5,0.5);
\draw [line width=0.8pt] (-4.5,0.5)-- (-4.5,0.);
\draw [line width=0.8pt] (-5.,0.)-- (-4.5,0.);
\draw [line width=0.8pt] (-4.5,0.)-- (-4.5,0.5);
\draw [line width=0.8pt] (-4.5,0.5)-- (-5.,0.5);
\draw [line width=0.8pt] (-5.,0.5)-- (-5.,0.);
\draw (-5.264452061755743,-0.7324064287473593) node[anchor=north west] {$\mathcal{B}_{m-2}$};
\draw (-4.124163133741892,-1.1315075535522072) node[anchor=north west] {$\mathcal{B}_{m-1}$};
\draw (-1.0311294165043192,-2.2290356467655386) node[anchor=north west] {$\mathcal{B}_m$};
\draw (-2.0573894517167854,-1.5591159015574012) node[anchor=north west] {$A$};
\draw (-2.784323643325616,-1.5591159015574012) node[anchor=north west] {$v_1$};
\draw (-1.7438099965129763,-2.0152314727629417) node[anchor=north west] {$v_1'$};
\draw (-1.5585130457107255,-0.6183775359459742) node[anchor=north west] {$v_2$};
\draw (-2.499251411322153,-0.6041239243458011) node[anchor=north west] {$g$};
\draw (-0.16165910889375745,-2.072245919163634) node[anchor=north west] {$v_3$};
\draw (-1.5157522109102062,-1.0744931071515147) node[anchor=north west] {$h$};
\begin{scriptsize}
\draw [fill=black] (-4.317902983809103,-0.1473382036869498) circle (0.5pt);
\draw [fill=black] (-4.214255071138008,-0.2604086538735984) circle (0.5pt);
\draw [fill=black] (-4.091762083435806,-0.373479104060247) circle (0.5pt);
\draw [fill=black] (0.25202771123461104,-0.8163383672912873) circle (0.5pt);
\draw [fill=black] (0.25202771123461104,-0.6561552295268684) circle (0.5pt);
\draw [fill=black] (0.628929211856773,-1.2309300179756655) circle (0.5pt);
\draw [fill=black] (0.7702672745900837,-1.2309300179756655) circle (0.5pt);
\draw [fill=black] (0.9304504123545025,-1.2215074804601114) circle (0.5pt);
\draw [fill=black] (0.25202771123461104,-0.47712701673134156) circle (0.5pt);
\draw [fill=uuuuuu] (-2.5,-1.5) circle (1.5pt);
\draw [fill=uuuuuu] (-1.5,-2.) circle (1.5pt);
\draw [fill=uuuuuu] (-1.5,-1.) circle (1.5pt);
\draw [fill=uuuuuu] (0.,-2.) circle (1.5pt);
\draw [fill=uuuuuu] (-2.5,-1.) circle (1.5pt);
\draw [fill=uuuuuu] (-1.5,-1.5) circle (1.5pt);
\end{scriptsize}
\end{tikzpicture}
 	\caption{}
 	\label{v_1 and v_1' aren't on the same ,ladder}
 \end{figure}

\noindent Then:
$$ f=f^+-f^-=\frac{f^+}{x_{v_1}x_{v_2}}(x_{v_1}x_{v_2}-x_gx_h)+\frac{f^+}{x_{v_1}x_{v_2}}x_gx_h-f^-.$$
Since $[v_1,v_2]$ is an inner interval of $\cP$, then $x_{v_1}x_{v_2}-x_gx_h \in I_{\cP}\subseteq J_{\cP}$. We set $\tilde{f}=\frac{f^+}{x_{v_1}x_{v_2}}x_gx_h-f^-$, $f_1=\frac{f^+}{x_{v_1}x_{v_2}}x_gx_h$ and $f_2=f^-$, so $\tilde{f}=f_1-f_2$. We observe that $\tilde{f}\in J_{\cP}$, $x_{v_3}x_h$ divides $f_1$ and $x_{v_1'}$ divides $f_2$. Since $v_3,h\in V^+_{\tilde{f}}$ and $v_1'\in V^-_{\tilde{f}}$, by Lemma \ref{Lemma shikama1}, we have that $\tilde{f}$ is redundant in $J_{\cP}$. Then $f$ is redundant in $J_{\cP}$, that is a contradiction. 
 \end{enumerate}
Summarizing, in $f$ there are not any variables associated to any vertices of $L_{\cB}$. We denote by $b_i$ the upper right corner of the cell $A_i$, for all $i=1,\dots,n$. We prove that in $f$ there is no variable associated to a vertex in $\{b_2,\dots,b_n\}$. We suppose that there exists $i\in\{2,\dots,n\}$ such that $x_{b_i}$ divides $f^+$. Let $V_{b_i}$ be the maximal vertical edge interval of $\cP$ such that $b_i\in V_{b_i}$. Since $b_{i}\in V^+_f$, there exists a vertex $v\in V_{b_i}\backslash\{b_i\}$, such that $x_v$ divides $f^-$. For the structure of $\cP$, the vertex $v$ belongs to $L_{\cB}$ and $v\in V^-_{f}$, that is a contradiction. Now we set $\cF=L_{\cB}\cup \{b_2,\dots,b_n\}$. In conclusion, in $f$ there are only variables $x_v$, such that $v\in V(\cP)\backslash \cF$. We denote by $\cP'$ the simple polyomino consisting of the cells of $\cP$ that do not have a vertex in $L_{\cB}$, and by $I_{\cP'}$ the polyomino ideal associated to $\cP'$. By Theorem 2.2 in \cite{Simple are prime}, we have that $J_{\cP'}=I_{\cP'}$. Moreover $f$ is a binomial in $J_{\cP'}$. Since $J_{\cP'}\subset J_{\cP}$ and $f$ is irredundant in $J_{\cP}$, then $f$ is irredundant in $J_{\cP'}$. It follows that $f$ is irredundant in $I_{\cP'}$, that is a contradiction. In conclusion we have $J_{\cP}\subseteq I_{\cP}$. 
\end{proof}
\begin{coro}\label{P closed path ladder then prime}
		Let $\cP$ be a closed path with a ladder of $m$ steps ($m>2$). Then $I_{\cP}$ is prime.
\end{coro}

\section{Primality of closed paths and zig-zag walks}\label{Section main result}
Let $\cP$ be a polyomino. In \cite{Trento} the authors have shown that if $I_{\cP}$ is prime then $\cP$ contains no zig-zag walks and they have conjectured that it is a sufficient condition for the primality of $I_{\cP}$. We recall that the rank of $\cP$, denoted by $\rank(\cP)$, is the  number of the cells of $\cP$. Using a computational method, they have shown that the conjecture is verified for $\rank(\cP)\leq 14$. Here we prove that the conjecture is true for the class of closed paths.

\begin{prop}\label{Ha zig zag2}
	Let $\cP$ be a closed path and suppose that $\cP$ has no zig-zag walks. Then $\cP$ has an L-configuration or a ladder with at least three steps.
	\label{pathNoZig}
	\end{prop}
\begin{proof}
 The structure of $\cP$ assures that there exists at least a sequence of distinct inner intervals $I_1,\ldots,I_\ell$ such that $|I_i\cap I_{i+1}|=1$ for all $i=1,\ldots,\ell-1$ and $|I_\ell\cap I_1|=1$. 
Let $I_1\cap I_\ell=\{v_1=v_{\ell+1}\}$ and $I_i\cap I_{i+1}=\{v_{i+1}\}$ with $i\in \{1,\ldots,\ell-1\}$. Suppose that $v_\ell$ and $v_{1}$ are not in the same edge interval. After appropriate reflections or rotations, we can suppose that $I_\ell$ is a horizontal interval having $v_\ell$ and $v_{1}$ as diagonal corners. Let $\mathcal{B}$ be the maximal horizontal block of $\cP$ containing $\cP(I_\ell)$. We examine all possible different cases. 
  \begin{itemize}
  \item $\mathcal{B}$ contains at least three cells. We can suppose that $\cP$ has no L-configurations, otherwise we have finished. Then a part of the polyomino has the shape of Figure~\ref{fig:prop6.1A}(A), where $v_\ell \in\{a,b\}$ and $v_{1}\in\{c,d\}$. So we have a ladder with at least three steps. 
  \item $\mathcal{B}=\cP(I_\ell)$ and it contains exactly two cells. Then we are in the case of Figure~\ref{fig:prop6.1A}(B), where $v_\ell=a$ and $v_{1}=b$. We have again a ladder with at least three steps.
  \item $\mathcal{B}=\cP(I_\ell)$ is a cell. Under the assumption that $\cP$ has no L-configurations, we are in the case of  Figure~\ref{fig:prop6.1A}(C). In particular $\cP$ has a ladder with at least three steps.   
  \end{itemize}

\begin{figure}[h]
\subfloat[][]{\includegraphics[scale=0.8]{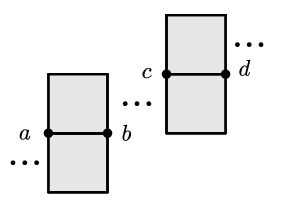}} \qquad\subfloat[][]{\includegraphics[scale=0.8]{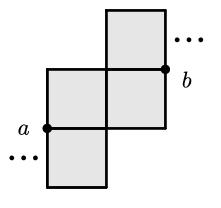}} 
\qquad\subfloat[][]{\includegraphics[scale=0.8]{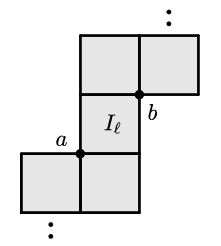}} 
\caption{}
\label{fig:prop6.1A}
\end{figure}

\noindent It remains to consider the case in which $I_\ell$ is a cell and $\mathcal{B}$ contains two cells. We prove that also in this case we obtain that $\cP$ contains an L-configuration or a ladder with at least three steps. After an appropriate reflection, we can reduce to the case in Figure~\ref{fig:prop6.1B}(A). Observe that if there is a cell in the direction West with respect to the cell A (that is the first cell of $I_{\ell-1}$), or in the direction North with respect to the cell D (that is the first cell of $I_1$), then $\cP$ has a ladder with at least three steps. So we can suppose that $\cP$ has an adjacent cell to $A$ in direction South and an adjacent cell to $D$ in direction East. In such a case we can define another sequence of intervals $I_1',\ldots,I_{\ell-1}'$, with $I_{\ell-1}'=I_{\ell-1} \cup B$, $I_{1}'=I_1 \cup C$ and $I_{i}'=I_{i}$ for $i\in \{2,\ldots,\ell-2\}$; in particular we denote $I_{1}'\cap I_{\ell-1}'=\{v_{1}'=v_{\ell}'\}$ and  $I_{i}'\cap I_{i+1}'=\{v_{i+1}'\}$ for $i\in \{1,\ldots,\ell-2\}$. So, we are in the situation of Figure~\ref{fig:prop6.1B}(B). Now suppose that $v_{1}'$ and $v_2'$ are not in the same edge interval. It is not difficult to see that in this case we are again in the situation of Figure~\ref{fig:prop6.1A}(A). So we can assume that $v_{1}'$ and $v_2'$ are in the same edge interval.  
\begin{figure}[h]
\subfloat[][]{\includegraphics[scale=0.8]{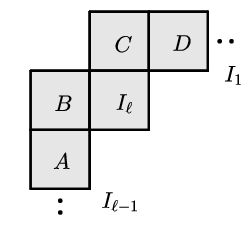}}  
\qquad\qquad\subfloat[][]{\includegraphics[scale=0.8]{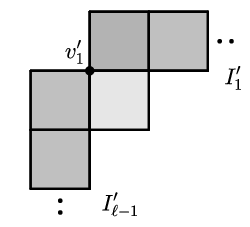}}
\caption{}
\label{fig:prop6.1B}
\end{figure}
\noindent The same conclusion can be obtained for the vertices $v_2'$ and $v_3'$ and so on. Therefore we can reduce the proof to the case that in the initial sequence of intervals $I_1,\ldots, I_\ell$ the vertices $v_i$ and $v_{i+1}$ belong to the same edge interval for every $i\in\{1,\ldots,\ell\}$. Since $\cP$ has no zig-zag walks, there exist $z_i$ and $z_j$ vertices of an inner interval $J$ of $\cP$, such that $v_i$ and $z_i$ are the diagonal or anti-diagonal corners of $I_i$, and $v_j$ and $z_j$ are the diagonal or anti-diagonal corners of $I_j$. Because of the structure of $\cP$ the only possibilities are $j=i+1$ or $j=i-1$. We can assume that $j=i+1$ and $v_{i+i}$ is a diagonal corner of $I_{i+1}$, so $v_i$ is an anti-diagonal corner of $I_i$. Let $\mathcal{B}_{i},\mathcal{B}_{i+1}$ be the maximal blocks of $\cP$ containing $\cP(I_i)$ and $\cP(I_{i+1})$ respectively. Observe that $\mathcal{B}_i$ and $\mathcal{B}_{i+1}$ are not both in horizontal or vertical position, since $J$ is an interval of $\cP$, that is a closed path. 
So we can assume that $\mathcal{B}_i$ is in vertical position and $\mathcal{B}_{i+1}$ is in horizontal position. Observe that each block has at least three cells; in particular we refer to Figure~\ref{fig:prop6.1} for the arrangement of this situation, observing that some appropriate cells with dashed lines must belong to the polyomino. In particular $\mathcal{B}_{i}\cup \mathcal{B}_{i+1}$ contains an L-configuration.
\begin{figure}[h!]
	\centering
	\includegraphics[scale=0.18]{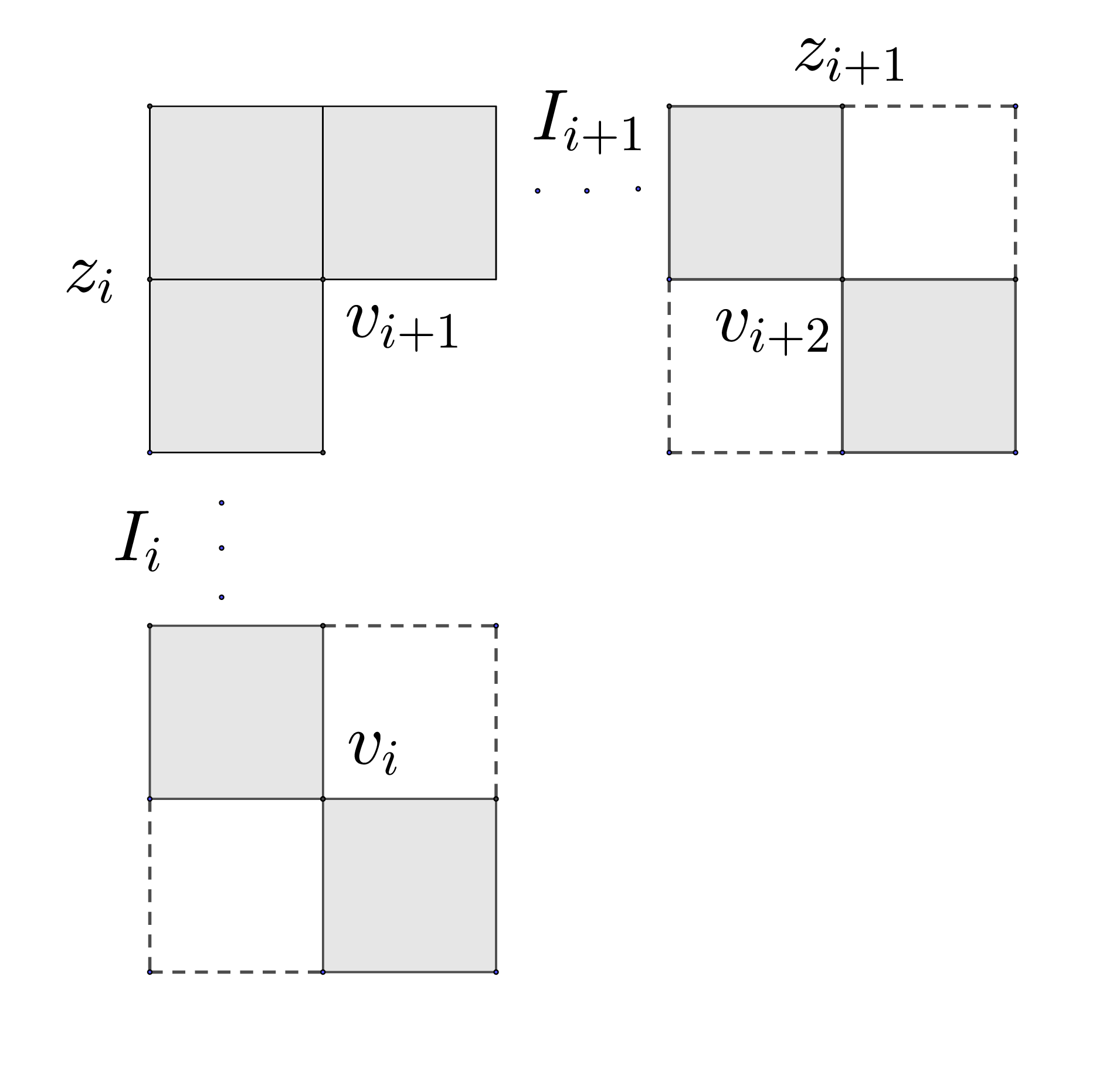}
	\caption{}
	\label{fig:prop6.1}
\end{figure}

\end{proof}

\noindent By Proposition~\ref{L-conf then no zig zag}, Proposition~\ref{ladder then no zig zag} and Proposition~\ref{pathNoZig} we deduce that having an L-configuration or a ladder with at least three steps is a necessary and sufficient condition in order to have no zig-zag walks for a closed path. 
Now we are ready to state and to prove the main result of this work.

\begin{thm}
	Let $\cP$ be a closed path. $I_{\cP}$ is prime if and only if $\cP$ contains no zig-zag walks.
\end{thm}
\begin{proof}
	The necessary condition is shown in \cite[Corollary 3.6]{Trento}. The sufficient one follows from the Proposition \ref{Ha zig zag2}, Corollary \ref{P closed path L cof then prime} and Corollary \ref{P closed path ladder then prime}.
\end{proof}

\section{Primality of other classes of polyominoes with paths}

Actually the arguments in the proofs of the results contained in this work can provide also the primality for a larger class of polyominoes. First of all, we introduce some useful definitions and notions.\\
We call a \emph{L-triomino} any polyomino consisting of three cells not aligned; for instance see Figure~\ref{FigTrimino} (A). Referring to the figure, we call \emph{hooking vertices} the vertices $a$ and $b$. Moreover we call \emph{hooking edges} with respect to $a$ (resp. $b$) the couple of edges of $A$ (resp. $B$) that intersect at $a$ (resp. at $b$).\\ 
\noindent A polyomino $\cC$ is called an \textit{open path} if it is a sequence of two or more cells $A_1,\dots,A_n$ such that:
	\begin{enumerate}
		\item $A_i\cap A_{i+1}$ is a common edge, for all $i=1,\dots,n-1$;
		\item $A_i\neq A_j$, for all $i\neq j$ and $i,j\in \{1,\dots,n\}$;
		\item If $n>2$, then $V(A_i)\cap V(A_j)=\emptyset$ for all $i\in\{1,\dots,n-2\}$ and for all $j\notin\{i,i+1,i+2\}$.
	\end{enumerate}
The edges of $A_1$ (resp. $A_n$), which do not belong to $E(A_2)$ (resp. $E(A_{n-1})$), are called \textit{free edges}.

\begin{figure}[h]
	\subfloat[][An L-triomino.]{\includegraphics[scale=0.9]{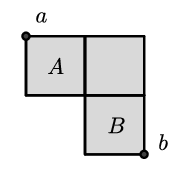}}\qquad \qquad\qquad
	\subfloat[][An open path.]{\includegraphics[scale=0.9]{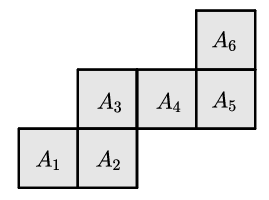}}  
	\caption{}
	\label{FigTrimino}
\end{figure}

\begin{defn} \rm
Let $\mathcal{S}$ be a simple polyomino, $\mathcal{C}:A_1,\ldots,A_n$ be an open path and $\mathcal{T}_1$ and $\mathcal{T}_2$ be two L-triominoes. Moreover we denote by $a_1$, $b_1$ the hooking vertices of $\mathcal{T}_{1}$ and by  $a_2$, $b_2$ the hooking vertices of $\mathcal{T}_{2}$. We denote by $\mathcal{P}(\mathcal{S},\mathcal{C})$ a polyomino satisfying the following conditions:
\begin{enumerate}
\item $\mathcal{P}(\mathcal{S},\mathcal{C})=\mathcal{S}\cup \mathcal{C}\cup \mathcal{T}_{1}\cup \mathcal{T}_{2}$.
\item $V(\mathcal{S})\cap V(\mathcal{C})=\emptyset$ and $V(\mathcal{T}_1)\cap V(\mathcal{T}_2)=\emptyset$.
\item $E(\mathcal{S})\cap E(\mathcal{T}_1)=\{V_1\}$ where $V_1$ is a hooking edge with respect to $a_1$  and $E(\mathcal{S})\cap E(\mathcal{T}_2)=\{V_2\}$ where $V_2$ is a hooking edge with respect to $a_2$.
\item $E(\mathcal{C})\cap E(\mathcal{T}_1)=\{W_1\}$ where $W_1\in E(A_1)$ and it is is a hooking edge with respect to $b_1$, and $E(\mathcal{C})\cap E(\mathcal{T}_2)=\{W_2\}$ where $W_2\in E(A_n)$ and it is a hooking edge with respect to $b_2$.
\item $|V(\cC)\cap V(\mathcal{T}_1)|=|V(\cS)\cap V(\mathcal{T}_1)|=|V(\cS)\cap V(\mathcal{T}_2)|=|V(\cC)\cap V(\mathcal{T}_2)|=2$.
\end{enumerate}
\begin{figure}[h]
\centering
\includegraphics[scale=1]{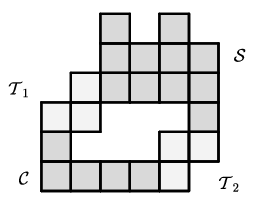}
\caption{An example of $\cP(\cS,\cC)$.}
\end{figure}
\end{defn}

\begin{rmk}\rm
    According to Proposition \ref{P is a not closed path}, it is easy to prove that a polyomino $\mathcal{P}(\mathcal{S},\mathcal{C})$, where $\mathcal{S}$ is a simple polyomino and $\mathcal{C}$ is an open path, is a non-simple polyomino and has only one hole. Moreover the polyomino consisting of all the cells of $\mathcal{P}(\mathcal{S},\mathcal{C})$, except two or more adjacent cells of $\cC$, is a simple polyomino.
    \end{rmk}

\begin{thm}
Let $\mathcal{P}=\mathcal{P}(\mathcal{S},\mathcal{C})$ be a polyomino with $\mathcal{S}$ a simple polyomino and $\mathcal{C}$ an open path. Suppose that $\mathcal{C}$ contains an L-configuration or a ladder with at least three steps. Then $I_{\cP}$ is a prime ideal.
\end{thm}

\begin{proof}
    If $\mathcal{C}$ contains an L-configuration then by defining the toric ideal as in Section~4 we obtain the claim following the same steps as in Proposition \ref{closed path L then I in J} and  Theorem \ref{proof1}, since the structure of $\cP$ allows it. If $\mathcal{C}$ contains a ladder with at least three steps the proof is similar, considering the toric ideal in Section~5, Proposition \ref{I in J ladder conf} and Theorem \ref{proof2}.
\end{proof}

\begin{rmk} \rm
Observe that if $\mathcal{C}$ contains an L-configuration or a ladder with at least three steps then $\mathcal{P}(\mathcal{S},\mathcal{C})$ has no zig-zag walks. The converse is not true (see Figure \ref*{ P(S,C) senza zig zag}), so it is an open question to ask what are the conditions allowing $\mathcal{P}(\mathcal{S},\mathcal{C})$ to have no zig-zag walks. In particular, we ask if the conjecture in \cite{Trento} is true also for polyominoes like $\mathcal{P}(\mathcal{S},\mathcal{C})$.
\begin{figure}[h!]
	\centering
\includegraphics[scale=1]{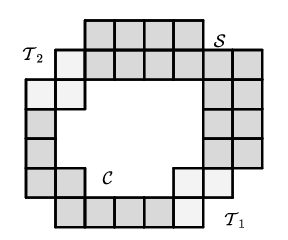}
	\caption{}
	\label{ P(S,C) senza zig zag}
\end{figure}
\end{rmk}

\noindent In \cite{Shikama} the author studied the polyomino ideal attached to a polyomino obtained by removing a convex polyomino from its ambient rectangle $\cR$. Our idea is to build a non-simple polyomino adding two open paths and a simple polyomino to a rectangle $\cR$. 
\begin{defn}\rm \label{def rectangle linked simple}
Let $\cR$ be a rectangle polyomino, associated to the interval $[(1,1),(m,n)]$, where $m\geq 4$ and $n\geq 2$.
Let $\cS$ be a simple polyomino, $\cP_1:C_1,\dots,C_t$ and $\cP_2:F_1,\dots,F_p$ be two open paths. A \textit{rectangle linked to a simple polyomino by two paths}, denoted by $\cP(\cR,\cP_1,S,\cP_2)$, is a polyomino satisfying the following conditions, after opportune reflections or rotations:
\begin{enumerate}
\item $\cP(\cR,\cP_1,S,\cP_2)=\cR\cup \cP_1\cup S\cup\cP_2$.
\item $V(\cS)\cap V(\cR)=\emptyset$ and $V(\cP_1)\cap V(\cP_2)=\emptyset$.	
\item The lower left corner of $C_1$ is $(1,n)$ and $V(\cP_1)\cap V(\cR)=\{(1,n),(2,n) \}$. 
\item $E(C_t)\cap E(\cS)=\{W\}$, where $W$ is a free edge of $C_t$, and $|V(\cP_1)\cap V(\cS)|=2$. 
\item $E(F_1)\cap E(\cS)=\{Z\}$, where $Z$ is a free edge of $F_1$, and $|V(\cP_2)\cap V(\cS)|=2$. 
	\item $E(F_p)\cap E(\cR)=\{V\}$, where $V$ is a free edge of $F_p$, and $|V(\cP_2)\cap V(\cR)|=2$.
	 
\end{enumerate}	
\end{defn}

\begin{rmk} \rm 
	On account of Proposition~\ref{P is a not closed path}, a rectangle linked to a simple polyomino by two paths is not a simple polyomino and it has a unique hole. Let us denote by $\cP_{\cR}$ the collection of cells of $\cR$, whose lower left corners are $(1,k)$ or $(2,k)$ for all $k=1,\dots,n-1$, and by $\cP_1^w$ the sequence of the first $w$ cells of $\cP_1$ for some $w\in\{1,\dots,t\}$. Then the polyomino consisting of all the cells of $\cP$ except the cells of $\cP_{\cR}\cup \cP_1^w$ is a simple polyomino.
\end{rmk}

\begin{defn}\rm
	Let $\cR$, $\cS$, $\cP_1$ and $\cP_2$ be as in the previous definition. A polyomino $\cP=\cP(\cR,\cP_1,S,\cP_2)$ is called an \textit{L-rectangle linked to a simple polyomino by two paths}, if 
	\begin{enumerate}
		\item  it satisfies all conditions in Definition \ref*{def rectangle linked simple};
		\item the lower left corner of $C_2$ is $(1,n+1)$;
		\item 
		let $V$ be the free	edge of $F_p$ such that $E(\cR)\cap E(\cP_2)=\{V\}$. Then $V\in \big\{ \{ (k,n),(k+1,n) \}:k=3\dots,m-1 \big \} \cup \big\{ \{ (m,l),(m,l+1) \}:l=1\dots,n-1  \big\} \cup \big\{ \{ (h,1),(h+1,1) \}:h=3\dots,m-1 \big \} $.
	\end{enumerate}
	Let $V_1$ and $V_2$ be the maximal vertical edge intervals of $\cP$, which contain respectively the vertices $(1,n)$ and $(2,n)$. Denote by $E_{V_1,V_2}$ the shortest maximal vertical edge interval between $V_1$ and $V_2$. Moreover, for all $k\in\{1,\dots,n\}$ let $H_k$ be the maximal horizontal edge interval containing $(1,k)$, $\cF_k$ the shortest one between $H_k$ and $H_{k+1}$ for each $k\in \{1,\dots,n-1\}$.  We call $\cP$ \textit{good} if the following cells belong to $\cP$:
	
	\begin{itemize}
	    \item  all cells having an edge in $E_{V_1,V_2}$ and lying between $V_1$ and $V_2$;
	    \item all cells having an edge in $\cF_k$ and lying between $H_k$ and $H_{k+1}$, for all $k\in \{1,\dots,n-1\}$. 
	\end{itemize}
\end{defn}
\begin{figure}[h]
	\centering
\includegraphics[scale=1]{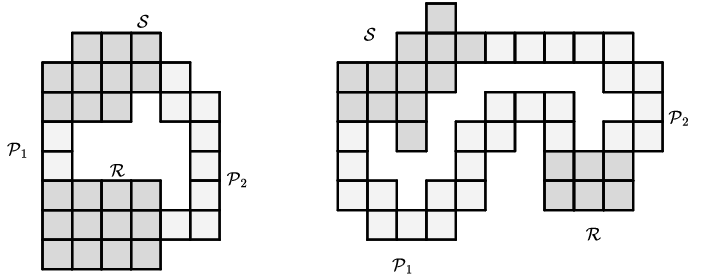}
	\caption{L-rectangles linked to a simple polyomino by two paths.}
	\label{L-rect}
\end{figure}
\noindent Referring to Figure \ref{Good rectangle example}, (A) and (B) are L-rectangles linked to a simple polyomino by two paths but not good, (C) is a good one, just as polyominoes in Figure \ref{L-rect} are.
\begin{figure}[h]
	\centering
\includegraphics[scale=1]{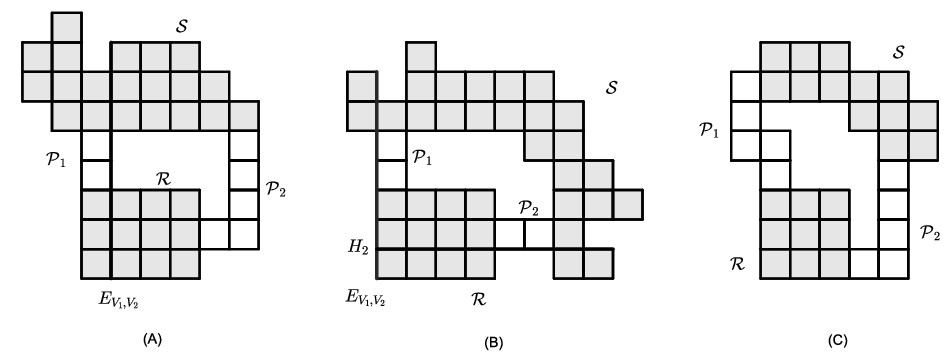}
	\caption{}
	\label{Good rectangle example}
\end{figure}
\begin{prop}
Let $\cP=\cP(\cR,\cP_1,S,\cP_2)$ be a good L-rectangle linked to a simple polyomino by two paths. Then $I_{\cP}$ is prime.
\end{prop}

\begin{proof}
	 We denote by $e$ the vertex $(2,n)$. We define the toric ideal $J_{\cP}$ as in Section \ref{Section L-conf toric}, where $V(A)$ is replaced by $A_{e}=\{v\in V(\cR):v\leq e\}$. By similar arguments as in Proposition \ref{closed path L then I in J} and Theorem \ref{Path with L conf is toric}, we have $I_{\cP}=J_{\cP}$, because of the good structure of $\cP$. 
\end{proof}

\begin{defn} \rm 
	A polyomino $\cP(\cR,\cP_1,S,\cP_2)$ is called a \textit{ladder-rectangle linked to a simple polyomino by two paths}, if 
		\begin{enumerate}
		\item  it satisfies all conditions in Definition \ref*{def rectangle linked simple};
		\item $\cP_1$ contains two maximal horizontal blocks $[C_1,C_s]$ and $[C_{s+1},C_q]$, where $2\leq s< s+1< q\leq t$, and the lower left corner of $C_{s+1}$ is the upper left one of $C_s$;
		\item the free edge $V$ of $F_p$ such that $E(\cR)\cap E(\cP_2)=\{V\}$ satisfies $V\in \big\{ \{ (k,n),(k+1,n) \}:k=3, \dots,m-1 \big \}$.
	\end{enumerate}
\end{defn}

\begin{figure}[h]
	\centering
\includegraphics[scale=1]{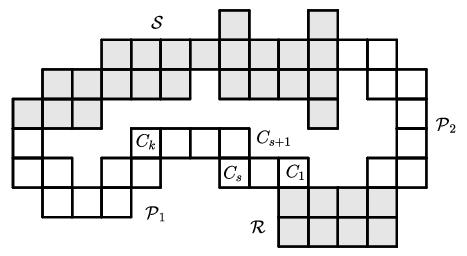}
	\caption{A ladder-rectangle linked to a simple polyomino by two paths.}
	\label{ladder rectangle}
\end{figure}

\begin{prop}
	Let $\cP=\cP(\cR,\cP_1,S,\cP_2)$ be a ladder-rectangle linked to a simple polyomino by two paths. Then $I_{\cP}$ is prime.
\end{prop}
\begin{proof}
 We denote by $e$ the vertex $(2,n)$ and by $a_i$ the lower left corner of the cell $C_i$ of $\cP_1$, for all $i\in \{1,\dots,s\}$. We define the toric ideal $J_{\cP}$ as in Section \ref*{Section toric ladder}, where $L_{e}=\{v\in V(\cR):v\leq e\}\cup \{a_2,\dots,a_s\}$. By similar arguments as in Proposition \ref{I in J ladder conf} and Theorem \ref{proof2}
  we deduce that $I_{\cP}=J_{\cP}$.
\end{proof}
\begin{rmk}\rm
	We observe that for the class of polyominoes $\cP(\cR,\cP_1,S,\cP_2)$ the following:
	\begin{enumerate}
		\item $\cP(\cR,\cP_1,S,\cP_2)$ is a good L-rectangle linked to a simple polyomino by two paths,
		\item $\cP(\cR,\cP_1,S,\cP_2)$ is a ladder-rectangle linked to a simple polyomino by two paths,
	\end{enumerate}
are sufficient conditions in order that it does not contain zig-zag walks. Necessary conditions to have no zig-zag walks and a positive answer to the conjecture in \cite{Trento} for polyominoes like $\cP(\cR,\cP_1,S,\cP_2)$ are open questions.
\end{rmk}

\subsection*{Acknowledgements} The authors wish to thank warmly Rosanna Utano, Ph.D. advisor of them, for introducing to this subject related to polyominoes and for her helpful suggestions and comments in writing of this work.


\begin{thebibliography}{99}
		
		\addcontentsline{toc}{chapter}{\bibname}
		
		
		\bibitem{conca1} A. Conca, Ladder Determinantal Rings, J. Pure Appl. Algebra, 98 119--134, 1995.
		
		\bibitem{conca2} A. Conca, Gorenstein ladder determinantal rings, J. London Math. Soc., 54, 453--474, 1996.
		
		\bibitem{conca3} A. Conca, J. Herzog, Ladder Determinantal Rings Have Rational Singularities, Adv. Math. 132, 120--147, 1997.
		
		\bibitem{golomb} S. W. Golomb, Polyominoes, puzzles, patterns, problems, and packagings, Second edition, Princeton University press, 1994.
		
		\bibitem{adiajent1} J. Herzog, T. Hibi, Ideals generated by adjacent 2-minors, J. Commut. Algebra
		4, 525--549, 2012.
		
	    \bibitem{2.n} J. Herzog, T. Hibi, F. Hreinsdoottir, T. Kahle, J. Rauh, Binomial edge ideals and conditional independence statements, Adv. Appl. Math. 45, 317--333, 2010.
	    
		\bibitem{Simple equivalent balanced} J. Herzog, S. S. Madani, The coordinate ring of a simple polyomino, Illinois J. Math., Vol. 58, 981--995, 2014.
		
		\bibitem{def balanced} J. Herzog, A. A. Qureshi, A. Shikama, Grobner basis of balanced polyominoes, Math. Nachr., Vol 288, no. 7, 775--783, 2015.
		
		\bibitem{Not simple with localization} T. Hibi, A. A. Qureshi, Nonsimple polyominoes and prime ideals, Illinois J. Math., Vol. 59, 391--398, 2015.
		
		\bibitem{adjent 2}S. Hosten, S. Sullivant, Ideals of adjacent minors, J. Algebra 277 , 615--642, 2004.
		
		
		\bibitem{Trento} C. Mascia, G. Rinaldo, F. Romeo, Primality of multiply connected polyominoes, Illinois J. Math., Vol. 64(3), 291--304, 2020.
		
		\bibitem{Trento2} C. Mascia, G. Rinaldo, F. Romeo, Primality of polyomino ideals by quadratic Gr\"obner basis. Accepted in Matematische Nachrichten, arXiv:2005.08758, 2020.
		
		\bibitem{adiajent3} H. Ohsugi, T. Hibi, Special simplices and Gorenstein toric rings. J. Combinatorial Theory Series A 113, 2006.
		
		\bibitem{Qureshi} A.A. Qureshi, Ideals generated by 2-minors, collections of cells and stack polyominoes, J. Algebra 357, 279--303, 2012.
		
		\bibitem{Simple are prime}  A. A. Qureshi, T. Shibuta, A. Shikama, Simple polyominoes are prime, J. Commut. Algebra 9, no. 3, 413--422, 2017.
				
		\bibitem{Shikama} A. Shikama, Toric representation of algebras defined by certain nonsimple polyominoes, J.
		Commut. Algebra, Vol. 10, 265--274, 2018.
		
		
	\end{thebibliography}
\end{document}